%% file: rtc.tex
\documentclass[final,letterpaper,11pt]{article}%
\usepackage[textwidth=6.0in,textheight=9in,centering,letterpaper]{geometry}
\usepackage{amsxtra,amscd}
\usepackage{graphicx}
\usepackage[amsmath,hyperref,thmmarks]{ntheorem}
\usepackage{natbib}
\usepackage{verbatim}
\usepackage{paralist}
\usepackage{amsmath}
\usepackage{amsfonts}
\usepackage{amssymb}
\usepackage{url}
\usepackage[scaled=0.92]{helvet}
\usepackage[notref,notcite]{showkeys}
\usepackage[all,cmtip]{xy}
\usepackage{diagxy}%
\theoremnumbering{arabic}
\theoremheaderfont{\scshape}
\RequirePackage{latexsym}
\theorembodyfont{\slshape}
\theoremseparator{}
\newtheorem{X}{X}[section]

\newtheorem{corollary}[X]{Corollary}

\newtheorem{lemma}[X]{Lemma}
\newtheorem{proposition}[X]{Proposition}
\newtheorem{theorem}[X]{Theorem}
\newtheorem{rc1}[X]{Rationality Conjecture}
\newtheorem{rc2}[X]{Weak Rationality Conjecture}
\theorembodyfont{\upshape}
\newtheorem{aside}[X]{Aside}
\newtheorem{definition}[X]{Definition}

\newtheorem{example}[X]{Example}

\newtheorem{remark}[X]{Remark}
\newtheorem{plain}[X]{}

\theorembodyfont{\small}

\theorembodyfont{\normalsize}

\theoremstyle{nonumberplain}
\theoremheaderfont{\sc}
\theorembodyfont{\normalfont}
\theoremsymbol{\ensuremath{_\Box}}
\RequirePackage{amssymb}
\newtheorem{proof}{Proof.}
\newtheorem{pf}{Proof}
\qedsymbol{\ensuremath{_\Box}}
\theoremclass{LaTeX}
\makeindex
\input{defna}

\begin{document}

\title{Rational Tate classes}
\author{J.S. Milne}
\date{April 29, 2008}
\maketitle

\begin{abstract}
\noindent In despair, as Deligne (2000)\nocite{deligne2000} put it, of proving
the Hodge and Tate conjectures, we can try to find substitutes. For abelian
varieties in characteristic zero, Deligne (1982) constructed a theory of Hodge
classes having many of the properties that the algebraic classes would have if
the Hodge conjecture were known. In this article I investigate whether there
exists a theory of \textquotedblleft rational Tate classes\textquotedblright%
\ on varieties over finite fields having the properties that the algebraic
classes would have if the Hodge and Tate conjectures were known. In
particular, I prove that there exists at most one \textquotedblleft
good\textquotedblright\ such theory. \newline v1 July 20, 2007. First version
on the web. \newline v2 November 7, 2007. Completely rewritten; shortened the
title. \newline v3 April 29, 2008. Submitted version.

\end{abstract}
\tableofcontents

\renewcommand{\thefootnote}{\fnsymbol{footnote}} \footnotetext{MSC2000:
114C25; 14K15; 11G10} \renewcommand{\thefootnote}{\arabic{footnote}}

\subsection*{Introduction}

In the absence of any significant progress towards a proof of the Hodge or
Tate conjectures, we can instead try to attach to each smooth projective
variety $X$ a graded $\mathbb{Q}{}$-algebra of cohomology classes having the
properties that the algebraic classes would have if one of the conjectures
were true. When the ground field $k$ is algebraically closed of characteristic
zero, every embedding $\sigma\colon k\hookrightarrow\mathbb{C}{}$ gives a
candidate for this $\mathbb{Q}{}$-algebra, namely, the $\mathbb{Q}{}$-algebra
of Hodge classes on $\sigma X$. The problem is then to show that this
$\mathbb{Q}{}$-algebra is independent of the embedding. \citet{deligne1982}
proves this for abelian varieties.

When the ground field is algebraically closed of characteristic $p\neq0$ the
problem is different. To each smooth projective variety one can attach a
$\mathbb{Q}{}_{l}$-algebra of Tate classes for every prime $l$ (including
$p$), and the problem is then to find a canonical $\mathbb{Q}{}$-structure on
these $\mathbb{Q}{}_{l}$-algebras. The purpose of this article is to examine
this problem for varieties over an algebraic closure $\mathbb{F}{}$ of
$\mathbb{F}{}_{p}$.

First I write down a list of properties that these $\mathbb{Q}{}$-structures
should have (and would have if the Hodge and Tate conjectures were known in
the relevant cases) in order to be a \textquotedblleft good theory of rational
Tate classes\textquotedblright. Then I prove (in \S 3) that there exists at
most one such theory (meaning \textit{exactly} one theory) on any class of
varieties for which the Frobenius maps are semisimple, for example, for
abelian varieties. Next I prove that the existence of such a theory would have
many of the same consequences for motives that the aforementioned conjectures
have. In addition, we recover the theorem (\cite{milne1999lm}) that the Hodge
conjecture for CM abelian varieties over $\mathbb{C}{}$ implies the Tate
conjecture for abelian varieties over $\mathbb{F}{}$.

The $\mathbb{Q}{}$-algebra generated by the divisor classes on an abelian
variety $A$ over $\mathbb{F}{}$ is a partial $\mathbb{Q}{}$-structure on the
cohomology of $A$. The Hodge classes on any CM lift of $A$ provide a second
partial $\mathbb{Q}{}$-structure. The rationality conjecture in \S 4 predicts
that these two partial $\mathbb{Q}{}$-structures are compatible. This
conjecture implies the existence of a good theory of rational Tate classes on
abelian varieties over $\mathbb{F}{}$, and is implied by the Hodge conjecture
for CM abelian varieties. However, since it is trivially true for simple
ordinary abelian varieties, it should be easier to prove than either the Hodge
or Tate conjectures.

With these results, it is possible to divide the Tate conjecture over
$\mathbb{F}{}$ into two parts:

\begin{enumerate}
\item There exists a good theory of rational Tate classes for smooth
projective varieties over $\mathbb{F}{}$ (which will be unique if it exists).

\item Every rational Tate class is algebraic.
\end{enumerate}

As noted, the Hodge conjecture for CM abelian varieties over $\mathbb{C}{}$
implies the rationality conjecture for abelian varieties. However, in some
respects the rationality conjecture is stronger than the Tate conjecture for
abelian varieties over $\mathbb{F}{}$, since it implies the Hodge standard
conjecture for rational Tate classes whereas the Tate conjecture over
$\mathbb{F}{}$ does not imply the Hodge standard conjecture for algebraic
classes. It seems to me that the rationality conjecture is the minimum that
is necessary to obtain a full understanding of Shimura varieties over
finite fields and, in particular, to prove the conjecture of Langlands and
Rapoport (1987)\nocite{langlandsR1987}.

\subsection*{Conventions}

All algebraic varieties are smooth and projective. Complex conjugation on
$\mathbb{C}{}$ is denoted by $\iota$. The symbol $\mathbb{F}{}$ denotes an
algebraic closure of $\mathbb{F}{}_{p}$, and $\ell$ always denotes a prime
$\neq p$. On the other hand, $l$ \emph{is allowed to equal} $p$. The degree of
an algebra over a field is its dimension as a vector space. We say that an
extension $K$ of a field $k$ \emph{splits} a semisimple $k{}$-algebra $E$ if
$E\otimes_{k}K$ is isomorphic to a product of matrix algebras over $K$. The
symbol $\simeq$ denotes a canonical (or a specifically given) isomorphism.

For a variety $X$, $H^{\ast}(X)=\bigoplus\nolimits_{i}H^{i}(X)$ and $H^{2\ast
}(X)(\ast)=\bigoplus\nolimits_{i}H^{2i}(X)(i)$; both are graded algebras over
the coefficient field of the cohomology.

Let $X$ be a variety over $\mathbb{F}{}$. A regular map $\pi\colon
X\rightarrow X$ is a \emph{Frobenius map} if it arises by extension of scalars
from the $q$-power Frobenius map on a model of $X$ over some subfield
$\mathbb{F}{}_{q}$ of $\mathbb{F}{}$. We let $\pi_{X}$ denote the family of
Frobenius maps of $X$. For $\ell\neq p$, $H_{\ell}^{\ast}(X)$ denotes the
\'{e}tale cohomology of $X$ with coefficients in $\mathbb{Q}{}_{\ell}$. An
element of $H_{\ell}^{2\ast}(X)(\ast)$ is an $\ell$\emph{-adic Tate class} if
it is fixed by some Frobenius map. The $\ell$-adic Tate classes on $X$ form a
graded $\mathbb{Q}{}_{\ell}$-subalgebra $\mathcal{T}_{\ell}(X)$ of $H_{\ell
}^{2\ast}(X)(\ast)$ of finite degree.

For a perfect field $k$ of characteristic $p$, $W(k)$ denotes the ring of Witt
vectors with coefficients in $k$, and $\sigma$ denotes the automorphism of
$W(k)$ that acts as $x\mapsto x^{p}$ on the residue field $k$. For a variety
$X$ over $k$, $H_{p}^{\ast}(X)$ denotes the crystalline cohomology of $X$ with
coefficients in the field of fractions $B(k)$ of $W(k)$. It is a graded
$B(k)$-algebra of finite degree with a $\sigma$-linear Frobenius map $F$. For
a variety $X$ over $\mathbb{F}{}$,
\[
\mathcal{T}_{p}^{r}(X)=\bigcup\nolimits_{X_{1}/\mathbb{F}{}_{q}}\left\{  a\in
H_{p}^{2r}(X_{1})(r)\mid Fa=a\right\}
\]
(union over the models $X_{1}/\mathbb{F}{}_{q}$ of $X$ over finite subfields
of $\mathbb{F}{}$). Then $\mathcal{T}{}_{p}(X)\overset{\textup{{\tiny def}}%
}{=}\bigoplus_{r}\mathcal{T}{}_{p}^{r}(X)$ is a graded $\mathbb{Q}{}_{p}%
$-algebra of finite degree whose elements are called the $p$\emph{-adic Tate
classes} on $X$.

The classes of the algebraic cycles on $X$ lie in $\mathcal{T}{}_{l}^{\ast
}(X)$, and the \emph{Tate conjecture for }$l$ states that their $\mathbb{Q}%
{}_{l}$-span is $\mathcal{T}{}_{l}^{\ast}(X)$.

\section{Preliminaries}

\subsection{Some linear algebra}

Throughout this subsection, $Q$ is a field.

\begin{plain}
\label{r5d}Let $V$ be a finite dimensional vector space over a field $Q$. Let
$\pi$ be an endomorphism of $V$, and let $V^{\pi}$ be the subspace of of $V$
of elements fixed by $\pi$. Then $\dim_{Q}V^{\pi}$ is at most the multiplicity
of $1$ as a root of the characteristic polynomial of $\pi$, and equals the
multiplicity if and only if $1$ is not a multiple root of the minimum
polynomial of $\pi$.
\end{plain}

\begin{plain}
\label{r5b}Let $V$ and $V^{\prime}$ be vector spaces over a field $Q$ in
duality by a pairing $\langle\,\,,\,\,\rangle\colon V\times V^{\prime
}\rightarrow Q$, and let $\pi$ and $\pi^{\prime}$ be endomorphisms of $V$ and
$V^{\prime}$ such that $\langle\pi v,\pi^{\prime}v^{\prime}\rangle=\langle
v,v^{\prime}\rangle$ for all $v\in V$ and $v^{\prime}\in V^{\prime}$. The
pairing%
\begin{equation}
v,v\mapsto\langle v,v^{\prime}\rangle\colon V^{\pi}\times V^{\prime\pi
^{\prime}}\rightarrow Q \label{e15}%
\end{equation}
is degenerate if and only if $1$ is a multiple root of the minimum polynomial
of $\pi$ on $V$.
\end{plain}

To see this, note that if $1$ is a multiple root of the minimum polynomial of
$\pi$, then there exists a nonzero $v\in V^{\pi}$ of the form $(\pi-1)w$ for
some $w\in V$, and%
\[
\langle v,v^{\prime}\rangle=\langle(\pi-1)w,v^{\prime}\rangle=\langle\pi
w,v^{\prime}\rangle-\langle w,v^{\prime}\rangle=\langle\pi w,\pi^{\prime
}v^{\prime}\rangle-\langle w,v^{\prime}\rangle=0
\]
for all $v^{\prime}\in V^{\prime\pi^{\prime}}$. Conversely, if $1$ is not a
multiple root of the minimum polynomial of $\pi$, then the same is true of
$\pi^{\prime}$ and the pairing (\ref{e15}) is obviously nondegenerate.

\begin{plain}
\label{r5e}Recall that an \emph{isocrystal} over perfect field $k$ is a
finite-dimensional $B(k)$-vector space $V$ together with a $\sigma$-linear
isomorphism $F\colon V\rightarrow V$. Let $k=\mathbb{F}_{p^{a}}$. Then
$\pi\overset{\textup{{\tiny def}}}{=}F^{a}$ is $B(k)$-linear. The following
statements are equivalent:

\begin{enumerate}
\item the isocrystal $(V,F)$ is semisimple;

\item the $\mathbb{Q}{}_{p}$-algebra $\End(V,F)$ is semisimple;

\item $\pi$ is a semisimple endomorphism of the $B(k)$-vector space $V$.
\end{enumerate}

\noindent(See, for example, \cite{milne1994}, 2.10.)
\end{plain}

\begin{plain}
\label{r5f}Let $(V,F)$ be an isocrystal over $k=\mathbb{F}{}_{p^{a}}$, and let
$V^{F}=\{v\in F\mid Fv=v\}$. Then $V^{F}$ is a $\mathbb{Q}{}_{p}$-subspace of
$V^{\pi}$ and $B(k)\otimes_{\mathbb{Q}{}_{p}}V^{F}\overset{\simeq
}{\longrightarrow}V^{\pi}$.
\end{plain}

Certainly, $V^{F}$ is a $\mathbb{Q}_{p}{}$-subspace of $V^{\pi}$, and we have
to prove that it is a $\mathbb{Q}{}_{p}$-structure on it. Obviously this is
true for a direct sum of isocrystals if and only if it is for each summand.
Therefore, we may assume that $(V,F)$ is indecomposable. According to the
structure theory of modules over the skew polynomial ring $A\overset
{\text{{\tiny def}}}{=}B(k)[F]$ (\cite{jacobson1943}, Chapter 3), there exists
a smallest $r$ for which $V^{r}\approx A/cA$ with $c$ in the centre of $A$.
The centre of $A$ is $\mathbb{Q}{}_{p}[F^{a}]$, and in fact $c=m(F^{a})$ with
$m$ a power of an irreducible polynomial. One can identify $m$ with the
minimum polynomial for $F^{a}$ as a $\mathbb{Q}{}_{p}$-linear map on $V$.
After the above remark, we may replace $V$ with $V^{r}$, and so assume that
$V=A/(m(F^{a}))$. Clearly, $V^{F}=V^{\pi}=0$ unless $m(T)$ is a power of
$T-1$, in which case a direct calculation shows that $B(k)\otimes
_{\mathbb{Q}{}_{p}}V^{F}\simeq V^{\pi}$.

\begin{plain}
\label{r5g}Let $(V,F)$ and $(V^{\prime},F)$ be isocrystals over $k=\mathbb{F}%
{}_{q}$, and suppose that $V$ and $V^{\prime}$ are in duality by a pairing
$\langle\,\,,\,\,\rangle\colon V\times V^{\prime}\rightarrow B(k)$ such that
$\langle Fv,Fv^{\prime}\rangle=\langle v,v^{\prime}\rangle$ for all $v\in V$
and $v^{\prime}\in V^{\prime}$. Then $\langle V^{F},V^{F^{\prime}}%
\rangle\subset B(k)^{F}=\mathbb{Q}{}_{p}$, and pairing%
\begin{equation}
v,v\mapsto\langle v,v^{\prime}\rangle\colon V^{F}\times V^{\prime
F}\rightarrow\mathbb{Q}{}_{p} \label{e21}%
\end{equation}
is degenerate if and only if $1$ is a multiple root of the minimum polynomial
of $\pi$ on $V$.
\end{plain}

That $\langle V^{F},V^{F^{\prime}}\rangle\subset\mathbb{Q}{}_{p}$ is obvious.
Statement (\ref{r5b}) shows that the pairing $V^{\pi}\times V^{\prime
\pi^{\prime}}\rightarrow B(k)$ is degenerate if and only if $1$ is a multiple
root of the minimum polynomial of $\pi$ on $V^{\prime}$, and so this follows
immediately from (\ref{r5f}).\bigskip

Let $Q_{0}$ be a subfield of $Q$. Let $W$ and $W^{\prime}$ be finite
dimensional $Q$-vector spaces, and let $\langle\,\,,\,\,\rangle\colon W\times
W^{\prime}\rightarrow Q$ be a bilinear pairing. Let $R$ and $R^{\prime}$ be
finite dimensional $Q_{0}$-subspaces of $W$ and $W^{\prime}$ such that
$\langle R,R^{\prime}\rangle\subset Q_{0}$:%
\[
\renewcommand{\arraystretch}{1.3}%
\begin{array}
[c]{ccccc}%
W & \times & W^{\prime} & \rightarrow & Q\\
\rotatebox{90}{$\subset$} &  & \rotatebox{90}{$\subset$} &  &
\rotatebox{90}{$\subset$}\\
R & \times & R^{\prime} & \rightarrow & Q_{0}.
\end{array}
\]
Consider the following statements.

\begin{description}
\item[\textnf{T:}] The map $f\otimes r\mapsto fr\colon Q\otimes_{Q_{0}%
}R\rightarrow W$ is surjective.

\item[\textnf{I:}] The map $f\otimes r\mapsto fr\colon Q\otimes_{Q_{0}%
}R\rightarrow W$ is injective.

\item[\textnf{S:}] The pairing $\langle\,\,,\,\,\rangle\colon W\times
W^{\prime}\rightarrow Q$ is left nondegenerate.

\item[\textnf{E:}] The pairing $\langle\,\,,\,\,\rangle\colon R\times
R^{\prime}\rightarrow Q_{0}$ is left nondegenerate.
\end{description}

\noindent There are also primed versions of these statements, for example,
T$^{\prime}$ is the statement \textquotedblleft$QR^{\prime}=W^{\prime}%
$\textquotedblright. \noindent Let $N$ be the left kernel of the pairing
$R\times R^{\prime}\rightarrow Q_{0}$, and consider the diagram:%
\[
\begin{CD}
Q\otimes_{Q_{0}}R @>b>> W @>c>> (W^{\prime})^{\vee}\\
@VVaV@.@VVdV\\
Q\otimes_{Q_{0}}(R/N) @>f>\textrm{injective}> Q\otimes_{Q_{0}}\Hom_{Q_0}(R^{\prime},Q_{0}) @>e>{\simeq}> (Q\otimes
_{Q_{0}}R^{\prime})^{\vee}.
\end{CD}
\]
Here $\left(  -\right)  ^{\vee}=\Hom_{Q}(-,Q)$, $b$ is the map $f\otimes
r\mapsto fr$, and $d$ is the dual of the similar map. The remaining maps are obvious.

\begin{proposition}
\label{r5a}\textnf{(a)} $\Ker(a)\supset\Ker(b)$, with equality if and only if
E is true.

\textnf{(b)} If E is true, then so is I.

\textnf{(c)} If S and T$^{\prime}$ are true, then so is E.

\textnf{(d)} dim$_{Q_{0}}(R/N)\leq\dim W$, with equality if and only if T and
E are true.
\end{proposition}

\begin{pf}
[following \cite{tate1994}, \S 2](a) We have $\Ker(a)\supset\Ker(b)$ because
$e\circ f$ is injective. Moreover, $b(\Ker(a))=Q\cdot N$, and so
$\Ker(a)\subset\Ker(b)$ if and only if $Q\cdot N=0$, i.e., $N=0$.

(b) We have E$\iff\Ker(a)=0\overset{\text{(a)}}{\implies}\Ker(b)=0\iff$I.

(c) If S and T$^{\prime}$ are true, then $c$ and $d$ are injective, and so
$\Ker(a)=\Ker(b)$.

(d) As $\Ker(b)\subset\Ker(a)$, we have a surjection%
\[
Q\cdot R\simeq(Q\otimes_{Q_{0}}R)/\Ker(b)\twoheadrightarrow(Q\otimes_{Q_{0}%
}R)/\Ker(a)\simeq Q\otimes_{Q_{0}}(R/N),
\]
and so $\dim_{Q}(Q\cdot R)\geq\dim_{Q_{0}}(R/N)$, with equality if and only if
$\Ker(a)=\Ker(b)$, i.e., $E$ holds. As $\dim_{Q{}}(Q\cdot R)\leq\dim W$, with
equality if and only if $T,$ this implies statement (d).
\end{pf}

Recall that the cup-product makes $H_{l}^{2\ast}(X)(\ast)$ into a graded
$\mathbb{Q}{}_{l}$-algebra (or $B(k)$-algebra if $l=p$), and that Poincar\'{e}
duality says that the product pairings%
\[
H_{l}^{2r}\left(  X\right)  (r)\times H_{l}^{2d-2r}(X)(d-r)\rightarrow
H_{l}^{2d}(X)(d)\overset{\langle\cdot\rangle}{\simeq}\mathbb{Q}{}_{l}%
\text{,}\quad d=\dim X,
\]
are nondegenerate for connected varieties $X$.

Let $X$ be a variety over $\mathbb{F}{}$. In this section and the next, we
let\footnote{For generalities on cohomology with ad\`{e}lic coefficients, see
\cite{milneR2004}, \S 2.}
\[
H_{\mathbb{A}{}}^{\ast}(X)=\left(  (\varprojlim\nolimits_{p\nmid m}H^{\ast
}(X_{\mathrm{et}},\mathbb{Z}{}/m\mathbb{Z}{}))\otimes_{\mathbb{Z}}\mathbb{Q}%
{}\right)  \times H_{p}^{\ast}(X).
\]
When $X$ is connected, there is an \textquotedblleft
orientation\textquotedblright\ isomorphism
\[
\langle\cdot\rangle\colon H_{\mathbb{A}{}}^{2\dim X}(X)(\dim X)\simeq
\mathbb{A}{}\overset{\text{{\tiny def}}}{=}\mathbb{A}{}^{p,\infty}\times
B(\mathbb{F}{}).
\]
For each $l$, there is a projection map $H_{\mathbb{A}{}}^{\ast}(X)\rightarrow
H_{l}^{\ast}(X)$.

\begin{theorem}
\label{r6}Let $X$ be a connected variety of dimension $d$ over $\mathbb{F}{}$,
and let $\mathcal{R}{}^{\ast}$ be a graded $\mathbb{Q}{}$-subalgebra of
$H_{\mathbb{A}{}}^{2\ast}(X)(\ast)$ of finite degree such that $\langle
\mathcal{R}^{d}\rangle\subset\mathbb{Q}{}$ and, for all $l$, the image of
$\mathcal{R}{}^{\ast}$ in $H_{l}^{2\ast}(X)(\ast)$ under the projection map is
contained in $\mathcal{T}{}_{l}^{\ast}(X)$. Fix an $r$. If, for some $l$,
\bquote ($\dagger$) the product pairings
\[
\mathcal{T}{}_{l}^{r}(X)\times\mathcal{T}{}_{l}^{d-r}(X)\rightarrow
\mathcal{T}{}_{l}^{d}(X)\simeq\mathbb{Q}{}_{l}%
\]
are nondegenerate and the images of $\mathcal{R}{}^{r}$ in $\mathcal{T}{}%
_{l}^{r}(X)$ and of $\mathcal{R}{}^{d-r}$ in $\mathcal{T}{}_{l}^{d-r}(X)$ span
them,\equote then this is true for all $l$; moreover, the pairing
$\mathcal{R}{}^{r}\times\mathcal{R}^{d-r}\rightarrow\mathbb{Q}$ is
nondegenerate and the map $\mathbb{Q}{}_{l}\otimes_{\mathbb{Q}{}}\mathcal{R}%
{}^{r}\rightarrow\mathcal{T}{}_{l}^{r}(X)$ is an isomorphism for all $l$.
\end{theorem}

\begin{proof}
Recall that, for any model $X_{1}/\mathbb{F}{}_{p^{a}}$ of $X$ over a finite
subfield of $\mathbb{F}{}$, the characteristic polynomial $P^{r}%
(X_{1}/\mathbb{F}{}_{p^{a}},T)\overset{\textup{{\tiny def}}}{=}\det
(1-\pi|H_{\ell}^{2r}(X)(r))$ of the Frobenius endomorphism $\pi$ of
$X_{1}/\mathbb{F}{}_{q}$ is independent of $\ell\neq p$, and moreover equals
the characteristic polynomial of $F^{a}$ acting on $H_{p}^{2r}(X_{1})(r)$
(\cite{katzM1974}). We let $m^{r}(X_{1})$ denote the multiplicity of $1$ as a
root of this polynomial, and we let $m^{r}=\max_{X_{1}}m^{r}(X_{1})$. Then
$\dim_{\mathbb{Q}{}_{l}}\mathcal{T}{}_{l}^{r}(X)\leq m^{r}$ (see \ref{r5d} and
\ref{r5f}).

Let $\mathcal{R}{}_{l}^{r}$ denote the image of $\mathcal{\mathcal{R}}^{r}$ in
$\mathcal{T}{}_{l}^{r}$, and let $N^{r}$ and $N_{l}^{r}$ denote the left
kernels of the pairings $\mathcal{R}{}^{r}\times\mathcal{R}{}^{d-r}%
\rightarrow\mathcal{R}{}^{d}$ and $\mathcal{R}{}_{l}^{r}\times\mathcal{R}%
{}_{l}^{d-r}\rightarrow\mathcal{R}{}_{l}^{d}$. Note that, because
$\mathcal{R}{}^{d-r}\rightarrow\mathcal{R}{}_{l}^{d-r}$ is surjective, the map
$\mathcal{R}{}^{r}\rightarrow\mathcal{R}{}_{l}^{r}$ sends $N^{r}$ into
$N_{l}^{r}$ and defines an isomorphism $\mathcal{R}{}^{r}/N^{r}\rightarrow
\mathcal{R}{}_{l}^{r}/N_{l}^{r}$.

We apply Proposition \ref{r5a} to the $\mathbb{Q}{}_{l}$-vector spaces
$\mathcal{T}_{l}^{r}(X)$ and $\mathcal{T}{}_{l}^{d-r}(X)$ and their
$\mathbb{Q}{}$-subspaces $\mathcal{R}{}_{l}^{r}$ and $\mathcal{R}{}_{l}^{d-r}%
$. Note that condition ($\dagger$) says that, for some $l$, statements $S$,
$T$, and $T^{\prime}$ hold, and hence also $E$ (by \ref{r5a}c).

For all $l$,%
\begin{equation}
\dim_{\mathbb{Q}{}}(\mathcal{R}{}^{r}/N^{r})=\dim_{\mathbb{Q}{}}(\mathcal{R}%
{}_{l}^{r}/N_{l}^{r})\overset{\text{(\ref{r5a}d)}}{\leq}\dim_{\mathbb{Q}{}%
_{l}}(\mathcal{T}{}_{l}^{r})\leq m^{r}\text{.} \label{e16}%
\end{equation}
Note that
\begin{equation}
\dim_{\mathbb{Q}{}}(\mathcal{R}{}_{l}^{r}/N_{l}^{r})=\dim_{\mathbb{Q}{}_{l}%
}(\mathcal{T}{}_{l}^{r})\overset{\text{(\ref{r5a}d)}}{\iff}\mathbb{Q}{}%
_{l}\cdot\mathcal{R}{}_{l}^{r}=\mathcal{T}{}_{l}^{r}\text{ and }N_{l}^{r}=0
\label{e22}%
\end{equation}
and that%
\begin{equation}
\dim_{\mathbb{Q}{}_{l}}(\mathcal{T}{}_{l}^{r})=m^{r}\overset{(\text{\ref{r5b}%
,\ref{r5g})}}{\iff}\text{the pairing }\mathcal{T}{}_{l}^{r}\times
\mathcal{T}_{l}^{d-r}\rightarrow\mathcal{T}{}_{l}^{d}\simeq\mathbb{Q}{}%
_{l}\text{ is nondegenerate.} \label{e23}%
\end{equation}

For those $l$ for which ($\dagger$) holds, the right hand statements in
(\ref{e22}) and (\ref{e23}) hold, and so equality holds throughout in
(\ref{e16}). Since the two end terms do not depend on $l$, equality holds
throughout in (\ref{e16}) for all $l$. Therefore the left hand statements in
(\ref{e22}) and (\ref{e23}) hold \emph{for all }$l$, and we deduce that

\begin{itemize}
\item the pairing $\mathcal{T}{}_{l}^{r}\times\mathcal{T}_{l}^{d-r}%
\rightarrow\mathbb{Q}{}_{l}$ is nondegenerate for all $l$,

\item the group $N_{l}^{r}=0$ for all $l$, and (by \ref{r5a}b)

\item the map $\mathbb{Q}{}_{l}\otimes_{\mathbb{Q}{}}\mathcal{R}{}_{l}%
^{r}\rightarrow\mathcal{T}_{l}^{r}(X)$ is an isomorphism for all $l$.
\end{itemize}

\noindent As $N^{r}$ maps into $N_{l}^{r}$ for all $l$ and the map
$\mathcal{R}^{\ast}{}\rightarrow\prod\nolimits_{l}H_{l}^{2\ast}(X)(\ast)$ is
injective, this implies that $N^{r}=0$ and so $\mathcal{R}{}^{r}%
\simeq\mathcal{R}{}_{l}^{r}$ for all $l$. Therefore $\mathbb{Q}{}_{l}%
\otimes_{\mathbb{Q}{}}\mathcal{R}{}^{r}\rightarrow\mathcal{T}{}_{l}^{r}(X)$ is
an isomorphism for all $l$ and $r$.
\end{proof}

\begin{remark}
\label{r6b}(a) In Proposition \ref{r5a}, it is not necessary to assume that
the maps $R\rightarrow W$ and $R^{\prime}\rightarrow W^{\prime}$ are injective.

(b) When applied to the $\mathbb{Q}{}$-subalgebra of $H_{\mathbb{A}{}}^{2\ast
}(X)(\ast)$ generated by algebraic classes, Theorem \ref{r6} extends Theorem
2.9 of \cite{tate1994} by allowing $\ell=p$.
\end{remark}

\subsection{An application of tannakian theory}

Throughout this section, $k$ is an algebraically closed field and $H_{W}$ is a
Weil cohomology theory on the algebraic varieties over $k$. By this I mean
that $H_{W}$ is a contravariant functor defined on the varieties over $k$,
sending disjoint unions to direct sums, and satisfying the conditions (1)--(4)
and (6) of \cite{kleiman1994}, \S 3, on connected varieties (finiteness,
Poincar\'{e} duality, K\"{u}nneth formula, cycle map, strong Lefschetz
theorem). The coefficient field of $H_{W}$ is denoted $Q$.

Let $\mathcal{S}{}$ be a class of algebraic varieties over $k$ satisfying the
following condition:\label{*}

\begin{quote}
(*) the projective spaces $\mathbb{P}{}^{n}$ are in $\mathcal{S}$, and
$\mathcal{S}{}$ is closed under passage to a connected component and under the
formation of products and disjoint unions.
\end{quote}

\noindent Let $Q_{0}$ be a subfield of $Q$, and for each $X\in\mathcal{S}{}$,
let $\mathcal{R}{}^{\ast}(X)$ be a graded $Q_{0}$-subalgebra of $H_{W}^{2\ast
}(X)(\ast)$ of finite degree. We assume the following:\label{R12}

\begin{enumerate}
\item[(R0)] for all connected $X\in\mathcal{S}{}$, the \textquotedblleft
orientation\textquotedblright\ isomorphism $H_{W}^{2\dim X}(X)(\dim X)\simeq
Q$ induces an isomorphism $\langle\cdot\rangle\colon\mathcal{R}{}^{\dim
X}(X)\simeq Q_{0}$;

\item[(R1)] for every regular map $f\colon X\rightarrow Y$ of varieties in
$\mathcal{S}{}$, $f^{\ast}\colon H_{W}^{2\ast}(Y)(\ast)\rightarrow$ $H_{W{}%
}^{2\ast}(X)(\ast)$ maps $\mathcal{R}{}^{\ast}(Y)$ into $\mathcal{R}{}^{\ast
}(X)$ and $f_{\ast}$ maps $\mathcal{R}{}^{\ast}(X)$ into $\mathcal{R}{}%
^{\ast+\dim Y-\dim X}(Y)$;\footnote{Whenever I write $\dim X$, I am implicitly
assuming that $X$ is equidimensional (and often that it is connected). I leave
it to the reader to make the necessary adjustments when it isn't.}

\item[(R2)] for every $X$ in $\mathcal{S}{}$, $\mathcal{R}{}^{1}(X)$ contains
the divisor classes.
\end{enumerate}

\noindent Because $\mathcal{R}{}^{\ast}(X)$ is closed under cup-products,
condition (R2) implies that the class of every point on $X$ lies in
$\mathcal{R}{}^{\dim X}(X)$, and so the isomorphism $\mathcal{R}{}^{\dim
X}(X)\simeq Q_{0}$ in (R0) is that sending the class of a point to $1$. The
cohomology class of the graph $\Gamma_{f}$ of any regular map $f\colon
X\rightarrow Y$ lies in $\mathcal{R}{}^{\dim Y}(X\times Y)$ because
$\Gamma_{f}=(\id_{X},f)_{\ast}(X)$ and so
\[
cl(\Gamma_{f})=(\id_{X},f)_{\ast}(cl(X))=(\id_{X},f)_{\ast}(1).
\]

The category of correspondences $\mathsf{C}(k)$ defined by $\mathcal{R}{}$ has
one object $hX$ for each $X\in\mathcal{S}{}$, and the morphisms from $X$ to
$Y$ are the elements of $\mathcal{R}{}^{\dim X}(X\times Y)$; composition of
morphisms is defined by the formula:%
\[
(f,g)\mapsto g\circ f=p_{XZ\ast}(p_{XY}^{\ast}f\cdot p_{YZ}^{\ast}%
g)\colon\mathcal{R}{}^{\dim X}(X\times Y)\times\mathcal{R}{}^{\dim Y}(Y\times
Z)\rightarrow\mathcal{R}{}^{\dim X}(X\times Z).
\]
This is a $Q{}_{0}$-linear category, and there is a contravariant functor from
the category of varieties in $\mathcal{S}{}$ to $\mathsf{C}(k{})$ sending $X$
to $hX$ and a regular map $f\colon Y\rightarrow X$ to the transpose of its
graph in $\mathcal{R}{}^{\dim X}(X\times Y)$.

Recall that the pseudo-abelian hull $\mathsf{C}^{+}$ of an additive category
$\mathsf{C}$ has one object $(x,e)$ for each object $x$ in $\mathsf{C}$ and
idempotent $e$ in $\End(x)$, and the morphisms from $(x,e)$ to $(y,f)$ are the
elements of the subgroup $f\circ\Hom(x,y)\circ e$ of $\Hom(x,y)$.

\begin{proposition}
\label{r6a}If the product pairings%
\begin{equation}
\mathcal{R}{}^{r}(X)\times\mathcal{R}{}^{\dim X-r}(X)\longrightarrow
\mathcal{R}{}^{\dim X}(X)\simeq Q_{0}{} \label{e13}%
\end{equation}
are nondegenerate for all connected $X\in\mathcal{S}{}$ and all $r\geq0$, then
$\mathsf{C}(k)^{+}$ is a semisimple abelian category.
\end{proposition}

\begin{pf}
[following \cite{jannsen1992}]An $f\in\mathcal{R}{}^{\dim X+r}(X\times Y)$
defines a linear map
\[
x\mapsto q_{\ast}(p^{\ast}x\cdot f)\cdot H_{W}^{\ast}(X)\rightarrow
H_{W}^{\ast+2r}(Y)(r).
\]
In particular, an element $f$ of $\mathcal{R}{}^{\dim X}(X\times X)$ defines
an endomorphism of $H_{W}^{\ast}(X)$. There is the following Lefschetz
formula: let $f,g\in\mathcal{R}{}^{\text{dim}X}(X\times X)$, and let $g^{t}$
be the transpose of $g$; then%
\[
\langle f\cdot g^{t}\rangle=\sum\nolimits_{i=0}^{2\dim X}(-1)^{i}\Tr(f\circ
g|H_{W}^{i}(X))
\]
(\cite{kleiman1968}, 1.3.6).

Let $f$ be an element of the ring $R(X)\overset{\text{{\tiny def}}}%
{=}\mathcal{R}{}^{\dim X}(X\times X)$. If $f$ is in the Jacobson
radical\footnote{Recall that the Jacobson radical of a ring $R$ is the set of
elements of $R$ that annihilate every simple $R$-module. It is a two-sided
ideal in $R$, which is nilpotent if $R$ is Artinian. A ring is semisimple if
and only if its Jacobson radical is zero.} of $R(X)$, then $f\cdot g^{t}$ is
nilpotent for all $g\in R(X)$, and so the Lefschetz formula shows that
$\langle f\cdot g\rangle=0$. Now (\ref{e13}) implies that $f=0$, and so the
ring $R(X)$ is semisimple. It follows that $e\cdot R(X)\cdot e$ is also
semisimple for any idempotent $e$ in $R(X)$. Thus $\mathsf{C}(k)^{+}$ is a
pseudo-abelian category such that $\End(x)$ is a semisimple $Q_{0}$-algebra of
finite degree for every object $x$, and this implies that it is a semisimple
abelian category (\cite{jannsen1992}, Lemma 2).
\end{pf}

The tensor product structure%
\[
hX\otimes hY\overset{\text{{\tiny def}}}{=}h(X\times Y)
\]
on $\mathsf{C}(k)$ extends to $\mathsf{C}(k)^{+}$, and with this structure
$\mathsf{C}(k)^{+}$ becomes a pseudo-abelian tensor category. The object
$h\mathbb{P}{}^{1}$ of $\mathsf{C}(k)$ decomposes into a direct sum
$\1\oplus\mathbb{L}{}$ in $\mathsf{C}(k)^{+}$, where $\mathbb{L}{}$ is (by
definition) the Lefschetz object. On inverting $\mathbb{L}{}$, we obtain the
category $\mathsf{M}(k)$ of false motives, which is a pseudo-abelian rigid
tensor category (\cite{saavedra1972}, VI 4.1.3.5). When the K\"{u}nneth
components of the diagonal of every variety $X$ in $\mathcal{S}{}$ lie in
$\mathcal{R}{}^{\dim X}(X\times X)$, they can be used to modify the
commutativity constraint on $\mathsf{M}(k)$ to obtain the category $\Mot(k)$
of (true) motives (ibid. VI 4.2.1.5). Every triple $(X,e,m)$ with
$X\in\mathcal{S}{}$, $e$ an idempotent in the ring $\mathcal{R}{}^{\dim
X}(X\times X)$, and $m\in\mathbb{Z}{}$, defines an object%
\[
h(X,e,m)\overset{\text{{\tiny def}}}{=}(hX,e)\otimes\mathbb{L}{}^{-m}%
\]
in $\Mot(k)$, and all objects of $\Mot(k)$ are isomorphic to an object of this form.

Now (\ref{r6a}) implies the following statement:

\begin{theorem}
\label{r6t}Assume that, for all connected $X\in\mathcal{S}{}$, the product
pairings (\ref{e13}) are nondegenerate and the K\"{u}nneth components of the
diagonal lie in $\mathcal{R}{}$. Then $\Mot(k)$ is a semisimple tannakian
category over $k$ with the $Q$-valued fibre functor $\omega_{W}\colon
h(X,e,m)\rightsquigarrow e(H_{W}^{\ast}(X))(m).$
\end{theorem}

Recall that, for a variety $X$ and any $n\geq0$, the K\"{u}nneth formula
provides an isomorphism%
\begin{equation}
H_{W}^{\ast}(X^{n})\simeq\bigotimes\nolimits^{n}H_{W}^{\ast}(X). \label{e14}%
\end{equation}
Therefore, every automorphism of the $Q$-vector space $H_{W}^{\ast}(X)$
defines an automorphism of $H_{W}^{\ast}(X^{n}).$

\begin{corollary}
\label{r6c}With the assumptions of the theorem, let $G_{X}$ be the largest
algebraic subgroup of $\GL(H_{W}^{\ast}(X))\times\GL(Q(1))$ fixing some
elements of $\bigoplus\nolimits_{n}\mathcal{R}{}^{\ast}(X^{n})$. Then%
\[
H_{W}^{2\ast}(X^{n})(\ast)^{G_{X}}\subset Q\cdot\mathcal{R}{}^{\ast}%
(X^{n})\text{ for all }n.
\]

\end{corollary}

\begin{proof}
Let $G=\underline{\Aut}^{\otimes}(\omega_{W})$. For every $Y\in\mathcal{S}{}$,
$G$ acts on $H_{W}^{2\ast}(Y)(\ast)$ and%
\[
H_{W}^{2\ast}(Y)(\ast)^{G}=Q\cdot\mathcal{R}{}^{\ast}(Y)
\]
(e.g., \cite{deligneM1982}). The image of $G$ in $\GL(H_{W}^{\ast}%
(X))\times\GL(Q(1))$ is contained in $G_{X}$, and the isomorphism (\ref{e14})
is $G$-equivariant, and so%
\[
H_{W}^{2\ast}(X^{n})(\ast)^{G_{X}}\subset H_{W}^{2\ast}(X^{n})(\ast
)^{G}=Q\cdot\mathcal{R}{}^{\ast}(X^{n}).
\]

\end{proof}

\subsection{Decomposition of the cohomology of an abelian variety over
$\mathbb{F}{}$}

Again let $H_{W}$ be a Weil cohomology theory with coefficient field $Q$. The
elements of the $\mathbb{Q}{}$-subalgebra of $H_{W}^{2\ast}(X)(\ast)$
generated by the divisor classes on a variety $X$ are called \emph{Lefschetz}
classes. A correspondence on a variety is said to be \emph{Lefschetz} if it is
defined by a Lefschetz class. For an abelian variety $A$, the $Q$-span of the
endomorphisms of $H_{W}^{2\ast}(A)(\ast)$ defined by Lefschetz classes
consists exactly of those commuting with the action of the Lefschetz group of
$A$. See \cite{milne1999lc}.

Let $A$ be an abelian variety over $k$ with sufficiently many endomorphisms,
i.e., such that $\End^{0}(A)$ contains an \'{e}tale subalgebra of degree
$2\dim A$ over $\mathbb{Q}{}$. The centre $C(A)$ of $\End^{0}(A)$ is a product
of CM-fields with possibly a copy of $\mathbb{Q}$, and so it has a
well-defined complex conjugation $\iota_{A}$. The Rosati involution of any
polarization of $A$ preserves each factor of $C(A)$ and acts on it as
$\iota_{A}$. The special Lefschetz group $S(A)$ of $A$ is the algebraic group
of multiplicative type over $\mathbb{Q}{}$ such that, for each $\mathbb{Q}{}%
$-algebra $R$,%
\[
S(A)(R{})=\{\gamma\in C(A)\otimes_{\mathbb{Q}{}}R\mid\gamma\cdot\iota
_{A}\gamma=1\}
\]
(\cite{milne1999lc}). It acts on $H_{W}^{\ast}(A)$, and when $Q$ splits
$S(A)$, we let $H_{W}^{\ast}(A)_{\chi}$ denote the subspace on which $S(A)$
acts through the character $\chi$ of $S(A)$.

Fix an isomorphism $Q\rightarrow Q(1)$ and use it to identify $H_{W}%
^{r}(A)(s)$ with $H_{W}^{r}(A)$.

\begin{lemma}
\label{r6z}Let $A$ be an abelian variety with sufficiently many endomorphisms,
and assume that $Q$ splits $C(A)$. Let $G$ be the centralizer (in the sense of
algebraic groups) of the image of $S(A)_{Q}$ in $\GL(H_{W}^{\ast}(A))$. Then
for each character $\chi$ of $S(A)$, the representation of $G$ on $H_{W}%
^{\ast}(A)_{\chi}$ is irreducible.
\end{lemma}

\begin{proof}
Let $H_{W}^{\ast}(A)=\bigoplus\nolimits_{\chi\in\Xi}H_{W}^{\ast}(A)_{\chi}$.
Then $G=\prod\nolimits_{\chi\in\Xi}\GL(H_{W}^{\ast}(A)_{\chi})$, and so the
statement is obvious.
\end{proof}

We next compute $X^{\ast}(S(A))$. Let $\Sigma=\Hom_{\mathbb{Q}{}\text{-alg}%
}(C(A),Q)$. If $A$ is a supersingular elliptic curve, then $C(A)=\mathbb{Q}{}$
and $S(A)=\mu_{2}$. In this case $X^{\ast}(C(A))=\mathbb{Z}{}/2\mathbb{Z}{}$.
If $A$ is simple, but not a supersingular elliptic curve, then $C(A)$ is a CM
field $E$ and $X^{\ast}(S(A))$ is the quotient of $\mathbb{Z}{}^{\Sigma}$ by
the group of functions $h$ such that $h(\sigma)=h(\sigma\circ\iota_{A})$ for
all $\sigma\in\Sigma$. For $h\in\mathbb{Z}{}^{\Sigma}$ , let
\[
f(\sigma)=h(\sigma)-h(\sigma\circ\iota_{A}),\quad\sigma\in\Sigma.
\]
Then $f$ is a map $f\colon\Sigma\rightarrow\mathbb{Z}$ such that
\begin{equation}
f(\sigma\circ\iota_{A})=-f(\sigma) \label{e20}%
\end{equation}
which depends only on the class of $h$ in $X^{\ast}(S(A))$, and every $f$
satisfying (\ref{e20}) arises from a unique $h\in X^{\ast}(S(A))$. For a
general $A$, let $I(A)$ be a set of representatives for the simple isogeny
factors of $A$. Then $S(A)\simeq\prod\nolimits_{B\in I(A)}S(B)$ and so%
\[
X^{\ast}(S(A))\simeq\bigoplus\nolimits_{B\in I(A)}X^{\ast}(S(B))\text{.}%
\]
It follows that $X^{\ast}(S(A))$ can be identified with the set of families
$f=(f(\sigma))_{\sigma\in\Sigma}$ such that:

\begin{itemize}
\item if $\sigma=\sigma\circ\iota_{A}$, then $f(\sigma)\in\mathbb{Z}%
{}/2\mathbb{Z}{}$;

\item if $\sigma\neq\sigma\circ\iota_{A}$, then $f(\sigma)\in\mathbb{Z}{}$ and
$f(\sigma\circ\iota)=-f(\sigma)$.
\end{itemize}

For $h\in\mathbb{Z}{}^{\Sigma}$, let $H(A)_{h}$ be the subspace of
$H_{W}^{\ast}(A)$ on which the torus $(\mathbb{G}_{m})_{E/\mathbb{Q}{}}$ acts
through the character $h$. Then there is a decomposition%
\[
H_{W}^{\ast}(A)=\bigoplus_{f\in X^{\ast}(S(A))}H(A)_{f}\quad\text{where
}H(A)_{f}=\bigoplus_{h\in\mathbb{Z}{}^{\Sigma}\text{, }h\mapsto f}H(A)_{h}.
\]
The cup-product pairing $H_{W}^{\ast}(A)\times H_{W}^{2\dim A-\ast
}(A)\rightarrow H_{W}^{2\dim A}\simeq Q$ is equivariant for the action of
$S(A)$, and so the subspaces $H(A)_{f}$ and $H(A)_{f^{\prime}}$ are orthogonal
unless $f+f^{\prime}=0$ in which case they are dual. Note that $\iota_{A}$
acts on $X^{\ast}(S(A))$ as $-1$.

\begin{theorem}
\label{r3}Let $A$ be an abelian variety over $k$ with sufficiently many
endomorphisms, and let $d=\dim A$. Let $\mathcal{R}{}^{\ast}$ be a graded
$\mathbb{Q}{}$-subalgebra of $H_{W}^{2\ast}(A)(\ast)$ of finite degree such
that $\mathcal{R}{}^{1}$ contains the divisor classes, $\mathcal{R}{}^{\ast}$
is stable under the endomorphisms of $H_{W}^{2\ast}(A)(\ast)$ defined by
Lefschetz correspondences, and $\dim_{\mathbb{Q}{}}\mathcal{R}{}^{\dim A}=1$.
\noindent If there exists a finite Galois extension $Q^{\prime}$ of $Q$
splitting $C(A)$ and admitting a $Q$-automorphism $\iota^{\prime}$ such that
$\sigma\circ\iota_{A}=\iota^{\prime}\circ\sigma$ for all homomorphisms
$\sigma\colon C(A)\rightarrow Q$, \noindent then the product pairings%
\[
\mathcal{R}{}^{r}\times\mathcal{R}{}^{d-r}\rightarrow\mathcal{R}{}^{d}%
\simeq\mathbb{Q}{}%
\]
are nondegenerate for all $r$, and the map%
\[
\mathcal{R}{}^{\ast}\otimes_{\mathbb{Q}{}}Q\rightarrow H_{W}^{2\ast}(A)(\ast)
\]
is injective.
\end{theorem}

\begin{proof}
The group $G$ acts on $H_{W}^{\ast}(A)$ by Lefschetz correspondences because
its action commutes with that of $S(A)$. Therefore, $Q\mathcal{R}{}^{\ast}$ is
stable under $G$. For $f\in X^{\ast}(S(A))$, let $H(A)_{f}=(Q^{\prime}%
\otimes_{Q}H_{W}^{\ast}(A))_{f}$. As $H(A)_{f}$ is a simple $G$-module (by
\ref{r6z} applied to $H_{W^{\prime}}=Q^{\prime}\otimes H_{W}$), the
intersection $Q^{\prime}\mathcal{R}{}^{\ast}\cap H(A)_{f}$ is either $0$ or
the whole of $H(A)_{f}$. Because $Q^{\prime}\mathcal{R}{}^{\ast}$ is stable
under the action of $\iota^{\prime}$,%
\[
H(A)_{f}\subset Q^{\prime}\mathcal{R}{}^{\ast}\implies\iota^{\prime}%
H(A)_{f}\subset Q^{\prime}\mathcal{R}{}^{\ast}\text{.}%
\]
But $\iota^{\prime}H(A)_{f}=H(A)_{-f}$, and so the cup-product pairings%
\[
Q^{\prime}\mathcal{R}{}^{r}\times Q^{\prime}\mathcal{R}{}^{d}\rightarrow
Q^{\prime}\mathcal{R}{}^{d}\simeq Q^{\prime}%
\]
are nondegenerate. Now we can apply (\ref{r5a}c) and (\ref{r5a}b).
\end{proof}

\begin{theorem}
[\cite{clozel1999}]\label{r5} For any abelian variety $A$ over $\mathbb{F}{}$,
$\ell$-adic homological equivalence coincides with numerical equivalence on a
set $S$ of primes $\ell$ of density $>0$.
\end{theorem}

\begin{proof}
Let $\mathcal{R}{}^{\ast}(A)$ be the $\mathbb{Q}{}$-subalgebra of $H_{\ell
}^{\ast}(A)$ generated by the algebraic classes, and let $E\subset\mathbb{C}%
{}$ be the smallest Galois extension of $\mathbb{Q}{}$ splitting $C(A)$. Then
$\sigma\circ\iota_{A}=\iota|E\circ\sigma$ for all homomorphisms $\sigma\colon
C(A)\rightarrow E$. Let $S$ be the set of primes $\ell$ such that $\iota|E$ is
the Frobenius element of some prime $\lambda$ of $E$ dividing $\ell$. Then the
hypotheses of Theorem \ref{r3} hold with $Q^{\prime}=E_{\lambda}$.
\end{proof}

\begin{remark}
\label{r4}Theorem \ref{r5} holds for $A^{n}$ with the \textit{same} set $S$
because $C(A)\simeq C(A^{n})$.
\end{remark}

\begin{aside}
\label{r4a}The proof of Clozel's theorem in this subsection simplifies that of
Deligne (see \cite{clozel2008}), who takes the group $G$ in Lemma \ref{r3} to
be the algebraic subgroup of $\GL(H_{W}^{\ast}(A))$ generated by
$\End(A)^{\times}$ and the group (isomorphic to $\SL_{2}$) given by Lefschetz
theory, and then proves the lemma by an explicit computation.
\end{aside}

\subsection{Quotients of tannakian categories}

I review some definitions and results from \cite{milne2007qtc}. Let $k$ be a
field, and let $\mathsf{T}$ be a tannakian category over $k$. A
\emph{tannakian subcategory} of $\mathsf{T}$ is a full $k$-linear subcategory
closed under the formation of subquotients, direct sums, tensor products, and
duals. In particular, it is \emph{strictly} full (i.e., it contains with any
object, every object in $\mathsf{T}$ isomorphic to the object). For any
subgroup $H$ of the fundamental group $\pi(\mathsf{T})$ of $\mathsf{T}$, the
full subcategory $\mathsf{T}^{H}$ of $\mathsf{T}$ whose objects are those on
which $H$ acts trivially is a tannakian subcategory of $\mathsf{T}$, and every
tannakian subcategory of $\mathsf{T}$ is of this form for a uniquely
determined subgroup of $\pi(\mathsf{T})$.

\textit{For simplicity, I assume throughout this subsection that }$\mathsf{T}%
$\textit{ has a commutative fundamental group. }Then $\pi(\mathsf{T})$ is an
ind-object in the subcategory $\mathsf{T}_{0}$ of $\mathsf{T}$ of trivial
objects (those isomorphic to the a direct sum of copies of the identity object
$\1$), and the equivalence of categories%
\begin{equation}
\Hom(\1,-)\colon\mathsf{T}_{0}\rightarrow\Vc_{k}\label{e019}%
\end{equation}
maps it to a pro-algebraic group in the usual sense. I often write
$\mathsf{T}^{\pi(\mathsf{T}\text{)}}$ or $\mathsf{T}^{\pi}$ for $\mathsf{T}%
_{0}$ and $\gamma^{\mathsf{T}}$ for the functor (\ref{e019}). Note that
$\gamma^{\mathsf{T}}$ is a $k$-valued fibre functor on $\mathsf{T}^{\pi}$ and
that, for any other $k$-valued fibre functor $\omega$ on $\mathsf{T}^{\pi}$,
there is a unique isomorphism $\gamma^{\mathsf{T}}\rightarrow\omega$ (because
$\underline{\Hom}^{\otimes}(\gamma^{\mathsf{T}},\omega)$ is a torsor for the
trivial group).

An exact tensor functor $q\colon\mathsf{T}\rightarrow\mathsf{Q}$ of tannakian
categories is a \emph{quotient functor} if every object of $\mathsf{Q}$ is a
subquotient of an object of the image of $q$. Then the full subcategory
$\mathsf{T}^{q}$ of $\mathsf{T}$ consisting of the objects that become trivial
in $\mathsf{Q}$ is a tannakian subcategory of $\mathsf{T}$, and
$X\rightsquigarrow\Hom(\1,qX)$ is a $k$-valued fibre functor $\omega^{q}$ on
$\mathsf{T}^{q}$. In particular, $\mathsf{T}^{q}$ is neutral. For any $X,Y$ in
$\mathsf{T}$,
\begin{equation}
\Hom(qX,qY)\simeq\omega^{q}(\underline{\Hom}(X,Y)^{H}), \label{e11}%
\end{equation}
where $H$ is the subgroup of $\pi(\mathsf{T})$ corresponding to $\mathsf{T}%
^{q}$. Every $k$-valued fibre functor $\omega_{0}$ on a tannakian subcategory
$\mathsf{S}$ of $\mathsf{T}$ arises from a well-defined quotient
$\mathsf{T}/\omega_{0}$ of $\mathsf{T}$. For example, when $\mathsf{T}$ is
semisimple, we can take $\mathsf{T}/\omega_{0}$ to be the pseudo-abelian hull
of the category with one object $qX$ for each object $X$ of $\mathsf{T}$ and
whose morphisms are given (\ref{e11}).

\begin{plain}
\label{q1}In summary, $(Q,q)\leftrightarrow\omega^{q}$ where\input{f1.tex}
\end{plain}

Let $q\colon\mathsf{T}\rightarrow\mathsf{Q}$ be a quotient functor, and let
$R$ be a $k$-algebra. An $R$-valued fibre functor $\omega$ on $\mathsf{Q}$
defines an $R$-valued fibre functor $\omega\circ q$ on $\mathsf{T}$, and the
(unique) isomorphism of fibre functors
\[
\Hom(\1,-)\rightarrow\omega|\mathsf{Q}_{0}%
\]
defines an isomorphism $a(\omega)\colon\omega^{q}\otimes_{k}R\rightarrow
(\omega\circ q)|\mathsf{T}^{q}$. Conversely, an $R$-valued fibre functor
$\omega^{\prime}$ on $\mathsf{T}$ together with an isomorphism $a\colon
\omega^{q}\otimes_{k}R\rightarrow\omega^{\prime}|\mathsf{T}^{q}$ defines a
fibre functor $\omega$ on $\mathsf{Q}$ whose action on objects is determined
by $\omega(qX)=\omega^{\prime}(X)$ and whose action on morphisms is determined
by
\[
\bfig\node a(0,500)[\Hom(qX,qY)]\node b(2000,500)[\Hom(\omega(qX),\omega
(qY))]\node c(0,0)[\omega^{q}(\underline{\Hom}(X,Y)^{H})\otimes R]\node
d(1000,0)[\omega^{\prime}(\underline{\Hom}(X,Y)^{H})]\node
e(2000,0)[\Hom(\omega^{\prime}X,\omega^{\prime}Y)^{\omega^{\prime}%
(H)}]\arrow/-->/[a`b;\omega]\arrow|r|/->/[a`c;(\ref{e11}%
)]\arrow[c`d;a]\arrow[d`e;\simeq]\arrow/^{(}->/[e`b;{}]\efig
\]

\begin{plain}
\label{q2}In summary, $\omega\leftrightarrow(\omega^{\prime},a)$ where \input{f2}
\end{plain}

\section{Rational Tate classes}

Throughout this section, $\mathcal{S}{}$ is a class of smooth projective
varieties over $\mathbb{F}{}$ satisfying the condition (*) (see p\pageref{*})
and containing the abelian varieties. The smallest such class will be denoted
$\mathcal{S}{}_{0}$. Thus, $\mathcal{S}{}_{0}$ consists of all varieties whose
connected components are products of abelian varieties and projective spaces.

\subsection{Definition}

\begin{definition}
\label{r7}A family $(\mathcal{R}{}^{\ast}(X))_{X\in\mathcal{S}{}}$ with each
$\mathcal{R}{}^{\ast}(X)$ a graded $\mathbb{Q}{}$-subalgebra of $H_{\mathbb{A}%
{}}^{2\ast}(X)(\ast)$ is a \emph{theory of rational Tate classes on}
$\mathcal{S}{}$ if it satisfies the following conditions:

\begin{enumerate}
\item[(R1)] for every regular map $f\colon X\rightarrow Y$ of varieties in
$\mathcal{S}{}$, $f^{\ast}$ maps $\mathcal{R}{}^{\ast}(Y)$ into $\mathcal{R}%
{}^{\ast}(X)$ and $f_{\ast}$ maps $\mathcal{R}{}^{\ast}(X)$ into
$\mathcal{R}^{\ast}(Y)$;

\item[(R2)] for every $X$ in $\mathcal{S}{}$, $\mathcal{R}{}^{1}(X)$ contains
the divisor classes;

\item[(R4)] for every prime $l$ (including $l=p$) and every $X$ in
$\mathcal{S}{}$, the projection map $H_{\mathbb{A}{}}^{\ast}(X)\rightarrow
H^{\ast}(X,\mathbb{Q}{}_{l})$ induces an isomorphism $\mathcal{R}{}^{\ast
}(X)\otimes_{\mathbb{Q}}\mathbb{Q}{}_{l}\rightarrow\mathcal{T}{}_{l}^{\ast
}(X)$.
\end{enumerate}
\end{definition}

\noindent Condition (R4) says that $\mathcal{R}{}^{\ast}(X)$ is simultaneously
a $\mathbb{Q}{}$-structure on each of the $\mathbb{Q}{}_{l}$-spaces
$\mathcal{T}{}_{l}^{\ast}(X)$ of Tate classes (including for $l=p$).
\noindent\noindent The elements of $\mathcal{R}{}^{\ast}(X)$ are called the
\emph{rational Tate classes} on $X$ for the theory $\mathcal{R}{}$.

For any $X$ in $\mathcal{S}{}$, let $\mathcal{A}{}^{\ast}(X)$ denote the
$\mathbb{Q}{}$-subalgebra of $H_{\mathbb{A}{}}^{2\ast}(X)(\ast)$ generated by
the algebraic classes. Then $\mathcal{A}{}^{\ast}(X)$ is a graded
$\mathbb{Q}{}$-algebra, and the family $(\mathcal{A}{}^{\ast}(X))_{X\in
\mathcal{S}{}}$ satisfies (R1) and (R2) of the definition. It is a theory of
rational Tate classes on $\mathcal{S}{}$ if the Tate conjecture holds for all
$X\in\mathcal{S}{}$ and numerical equivalence coincides with homological
equivalence for one (hence all) $l$.

\subsection{Properties of a theory of rational Tate classes}

\noindent Let $\mathcal{R}{}^{\ast}$ be a theory of rational Tate classes on
$\mathcal{S}{}$.\vspace{-5pt}\medskip

\begin{plain}
\label{r8} For every $X$ in $\mathcal{S}{}$, $\mathcal{R}{}^{\ast}(X)$ is a
$\mathbb{Q}{}$-algebra of finite degree. \noindent Indeed, for each $l$,
$\mathcal{R}{}^{\ast}(X)$ is a $\mathbb{Q}{}$-structure on the $\mathbb{Q}%
{}_{l}$-algebra $\mathcal{T}{}_{l}^{\ast}(X)$, which has finite degree.
\end{plain}

\begin{plain}
\label{r9} When $X$ is connected, there is a unique isomorphism $\mathcal{R}%
{}^{\dim X}(X)\rightarrow\mathbb{Q}$ sending the class of any point to $1$.
\noindent To see this, note that (R2) implies that $\mathcal{R}{}^{\ast}(X)$
contains all Lefschetz classes, and that the class of every point is
Lefschetz. Now apply (R4) noting that the similar statement is true for
$\mathcal{T}{}_{l}^{\dim X}(X)$ and $\mathbb{Q}{}_{l}$.
\end{plain}

\begin{plain}
\label{r10} For varieties $X,Y\in\mathcal{S}{}$, the maps $X\rightarrow
X\sqcup Y\leftarrow Y$ define an isomorphism%
\begin{equation}
\mathcal{R}{}^{\ast}(X\sqcup Y)\rightarrow\mathcal{R}{}^{\ast}(X)\oplus
\mathcal{R}{}^{\ast}(Y). \label{e10}%
\end{equation}
To see this, note that the isomorphism $H_{\mathbb{A}{}}^{2\ast}(X\sqcup
Y)\rightarrow H_{\mathbb{A}{}}^{2\ast}(X)\oplus H_{\mathbb{A}{}}^{2\ast}(Y)$
induces an injection (\ref{e10}), which becomes an isomorphism when tensored
with $l$.
\end{plain}

\begin{plain}
\label{r11} For any two varieties $X,Y$ in $\mathcal{S}{}$, there is a
$\mathbb{Q}{}$-algebra homomorphism%
\begin{equation}
x\otimes y\mapsto p^{\ast}x\cdot q^{\ast}y\colon\mathcal{R}{}^{\ast}%
(X)\otimes_{\mathbb{Q}{}}\mathcal{R}{}^{\ast}(Y)\rightarrow\mathcal{R}{}%
^{\ast}(X\times Y). \label{e1}%
\end{equation}

\end{plain}

\begin{plain}
\label{r12} A $c\in\mathcal{R}{}^{\dim X+r}(X\times Y)$ defines a linear map
\[
x\mapsto q_{\ast}(p^{\ast}x\cdot c))\colon\mathcal{R}{}^{\ast}(X)\rightarrow
\mathcal{R}{}^{\ast+r}(Y).
\]
In particular, Lefschetz correspondences map rational Tate classes to rational
Tate classes.
\end{plain}

\begin{plain}
\label{r13}Let $L$ be the Lefschetz operator on cohomology defined by a
hyperplane section of $X$. For $2r\leq\dim X$, the map
\[
L^{\dim X-2r}\colon\mathcal{R}{}^{r}(X)\rightarrow\mathcal{R}{}^{\dim X-r}(X)
\]
is injective. It is an isomorphism when $X$ is an abelian variety because then
the inverse map is a Lefschetz correspondence (\cite{milne1999lc}, 5.9).
\end{plain}

\begin{plain}
\label{r14} For any $n$, $\mathcal{R}{}^{\ast}(\mathbb{P}{}^{n})\simeq
\mathbb{Q}{}[t]/(t^{n+1})$ where $t$ denotes the class of any hyperplane in
$\mathbb{P}{}^{n}$, and, for any $X\in\mathcal{S}{}$, the map (\ref{e1}) is an
isomorphism%
\[
x\otimes y\mapsto p^{\ast}x\cdot q^{\ast}y\colon\mathcal{R}{}^{\ast}%
(X)\otimes\mathcal{R}{}^{\ast}(\mathbb{P}{}^{n})\simeq\mathcal{R}{}^{\ast
}(X\times\mathbb{P}{}^{n}).
\]

\end{plain}

\begin{plain}
\label{r15} Let $X$ be connected, and let $R(X)=\mathcal{R}{}^{\dim X}(X\times
X)$. Then $R(X)$ becomes a $\mathbb{Q}{}$-algebra with the product,%
\[
(f,g)\mapsto p_{13\ast}(p_{12}^{\ast}f\cdot p_{23}^{\ast}g)\colon\mathcal{R}%
{}^{\dim X}(X\times X)\times\mathcal{R}{}^{\dim X}(X\times X)\rightarrow
\mathcal{R}{}^{\dim X}(X\times X).
\]
It contains the graph of any regular map $f\colon X\rightarrow X$, and
$f\mapsto cl(\Gamma_{f})\colon\End(X)\rightarrow R(X)$ is a homomorphism. When
$X$ is not connected, we set $R(X)=\prod R(X_{i})$ where the $X_{i}$ are the
connected components of $X$.
\end{plain}

\subsection{Semisimple Frobenius maps}

Let $\mathcal{R}{}^{\ast}$ be a theory of rational Tate classes on
$\mathcal{S}{}$.

Recall that $\pi_{X}$ is the set of Frobenius maps of $X$. For $\pi\in\pi_{X}%
$, $\mathbb{Q}{}[\pi]$ denotes the $\mathbb{Q}{}$-subalgebra of $R(X)$
generated by the graph of $\pi$ (see \ref{r15}). For $N$ sufficiently
divisible, the $\mathbb{Q}{}$-algebra $\mathbb{Q}{}[\pi^{N}]$ depends only on
$\pi_{X}$, and is the algebra of least degree generated by an element of
$\pi_{X}$ --- we denote it $\mathbb{Q}{}\{\pi_{X}\}$. We say that $\pi_{X}$ is
\emph{semisimple}, or that $X$ \emph{has semisimple Frobenius maps}, if
$\mathbb{Q}{}\{\pi_{X}\}$ is semisimple, i.e., a product of fields. When
$\pi_{X}$ is semisimple, the Frobenius maps of $X$ act semisimply on all Weil
cohomology groups of $X$.

Weil (1948\nocite{weil1948}, Th\'{e}or\`{e}me 38) shows that the Frobenius
maps are semisimple if $\mathcal{S}{}=\mathcal{S}{}_{0}$.

\begin{proposition}
\label{r15a}Let $X$ be a connected variety of dimension $d$ in $\mathcal{S}{}%
$. If $\pi_{X}$ is semisimple, then $R(X)\overset{\textup{{\tiny def}}}%
{=}\mathcal{R}{}^{d}(X\times X)$ is a semisimple $\mathbb{Q}{}$-algebra with
centre $\mathbb{Q}{}\{\pi_{X}\}$, and the product pairings%
\begin{equation}
\mathcal{R}{}^{r}(X)\times\mathcal{R}{}^{d-r}(X)\rightarrow\mathcal{R}{}%
^{d}(X)\simeq\mathbb{Q}{} \label{e17}%
\end{equation}
are nondegenerate.
\end{proposition}

\begin{proof}
Fix an $\ell\neq p$, and let $\pi$ be a Frobenius element of $X$ such that
$\mathbb{Q}{}\{\pi_{X}\}=\mathbb{Q}{}[\pi]$. The K\"{u}nneth formula and the
Poincar\'{e} duality theorem give an isomorphism%
\[
H_{\ell}^{2d}(X\times X)(d)\simeq\End(H_{\ell}^{\ast}(X))
\]
(endomorphisms of $H_{\ell}^{\ast}(X)$ as a graded $\mathbb{Q}{}_{\ell}%
$-vector space), and the centralizer of $\mathbb{Q}{}_{\ell}[\pi]$ in this
$\mathbb{Q}{}_{\ell}$-algebra is $\mathcal{T}{}_{\ell}^{d}(X\times X)$.
Because $\mathbb{Q}{}[\pi]$ is semisimple, so also is $\mathbb{Q}{}_{\ell}%
[\pi]$, and it follows that $\mathcal{T}{}_{\ell}^{d}(X\times X)$ is a
semisimple $\mathbb{Q}{}_{\ell}$-algebra with centre $\mathbb{Q}{}_{\ell}%
[\pi]$. As $R(X)\otimes\mathbb{Q}{}_{\ell}\simeq\mathcal{T}{}_{\ell}%
^{d}(X\times X)$, it follows that $R(X)$ is semisimple with centre
$\mathbb{Q}{}[\pi]$.

The semisimplicity of $\pi_{X}$ implies that the pairings $\mathcal{T}{}%
_{l}^{r}(X)\times\mathcal{T}_{l}^{d-r}(X)\rightarrow\mathcal{T}{}_{l}%
^{d}\simeq\mathbb{Q}{}_{l}$ are nondegenerate (see \ref{r5b}), and so an
element of the left kernel of the pairing (\ref{e17}) maps to zero in
$\mathcal{T}{}_{l}^{\ast}(X)\subset H_{l}^{2\ast}(X)(\ast)$ for all $l$ (apply
\ref{r5a}c).
\end{proof}

\subsection{The Lefschetz standard conjecture}

Let $H_{W}$ be a Weil cohomology theory on the algebraic varieties over
$\mathbb{F}{}$, and let $(\mathcal{R}{}^{\ast}(X))_{X\in\mathcal{S}{}}$ be a
family of graded $\mathbb{Q}{}$-subalgebras of the $Q$-algebras $H_{W}^{2\ast
}(X)(\ast)$ satisfying (R1, R2, R4) --- we call this a \emph{theory of
rational Tate classes for} $H_{W}$. Let $X\in\mathcal{S}{}$ be connected, and
let $L$ be the Lefschetz operator defined by a smooth hyperplane section of
$X$. The following statements are the analogues for rational Tate classes of
the various forms of Grothendieck's Lefschetz standard conjecture
(\cite{grothendieck1968, kleiman1968, kleiman1994}):

\label{standard}

\begin{description}
\item[$A(X)$\textnf{:}] for $2r\leq d=\dim X$, $L^{d-2r}\colon\mathcal{R}%
{}^{r}(X)\rightarrow\mathcal{R}{}^{d-r}(X)$ is an isomorphism;

\item[$B(X)$\textnf{:}] the Lefschetz operator $\Lambda$ lies in
$\mathcal{R}{}^{\ast}(X\times X)$;

\item[$C(X)$\textnf{:}] the projectors $H_{W}^{\ast}(X)\rightarrow H_{W}%
^{i}(X)\subset H_{W}^{\ast}(X)$ lie in $\mathcal{R}{}^{\ast}(X\times X)$;

\item[$D(X)$\textnf{:}] the pairings $\mathcal{R}{}^{r}(X)\times\mathcal{R}%
{}^{d-r}(X)\rightarrow\mathcal{R}{}^{d}(X)\simeq\mathbb{Q}{}$ are nondegenerate.
\end{description}

\begin{theorem}
If statement $D(X)$ holds for all $X\in\mathcal{S}{}$, then $\pi_{X}$ is
semisimple for all $X\in\mathcal{S}{}$. Conversely, if $\pi_{X}$ is semisimple
for all $X\in\mathcal{S}{}$, then $A(X)$, $B(X)$, $C(X)$, and $D(X)$ hold for
all $X\in\mathcal{S}{}$ and all $L$.
\end{theorem}

\begin{proof}
Statement $D(X)$ implies that the $\mathbb{Q}{}$-algebra $R(X)\overset
{\textup{{\tiny def}}}{=}\mathcal{R}{}^{\dim X}(X\times X)$ is semisimple (see
\ref{r6a}), and therefore its centre $\mathbb{Q}{}\{\pi_{X}\}$ is semisimple.
Conversely, as in (\ref{r15a}), the semisimplicity of $\mathbb{Q}{}\{\pi
_{X}\}$ implies that $D(X)$ holds, and it is known that if $D(X)$ holds for
all $X$ in a set $\mathcal{S}{}$ satisfying (*), then so do $A(X)$, $B(X)$,
and $C(X)$ (e.g., \cite{kleiman1994}, 4-1, 5-1).
\end{proof}

\subsection{The category of motives for rational Tate classes}

Let $\mathcal{R}{}^{\ast}$ be a theory of rational Tate classes on
$\mathcal{S}{}$. As in \S 1, the category of correspondences $\mathsf{C}%
(\mathbb{F}{})$ has one object $hX$ for each $X\in\mathcal{S}{}$, and the
morphisms from $X$ to $Y$ are the elements of $\mathcal{R}{}^{\dim X}(X\times
Y)$.

\begin{proposition}
\label{r16}If $\pi_{X}$ is semisimple for every $X\in\mathcal{S}{}$, then the
pseudo-abelian hull of $\mathsf{C}(\mathbb{F}{})$ is a semisimple abelian category.
\end{proposition}

\begin{proof}
For $X\in\mathcal{S}{}$, $\End(hX)\simeq R(X)^{\mathrm{opp}}$, which
Proposition \ref{r15a} shows to be semisimple. Thus, the semisimplicity of the
Frobenius elements implies that the endomorphism algebras of the objects of
$\mathsf{C}(\mathbb{F}{})$ are semisimple $\mathbb{Q}{}$-algebras of finite
degree, and so $\mathsf{C}(\mathbb{F}{})$ is a semisimple abelian category by
\cite{jannsen1992}, Lemma 2.
\end{proof}

\begin{proposition}
\label{r19}For every $X$ in $\mathcal{S}{}$, the K\"{u}nneth components of the
diagonal are rational Tate classes.
\end{proposition}

\begin{proof}
In fact, they are polynomials in the graph of the Frobenius map with rational
coefficients (see, for example, \cite{katzM1974}, Theorem 2).
\end{proof}

The category of motives $\Mot(\mathbb{F}{})$ is obtained from $\mathsf{C}%
(\mathbb{F}{})$ by passing to the pseudo-abelian hull, inverting the Lefschetz
object, and using (\ref{r19}) to change the commutativity constraint. When the
Frobenius elements are semisimple, the article \cite{milne1994} can be
rewritten with the algebraic classes replaced by rational Tate classes. In
particular, we have the following result.

\begin{theorem}
\label{r20}If the Frobenius maps of the varieties in $\mathcal{S}{}$ are
semisimple, then the category $\Mot(\mathbb{F}{})$ is a semisimple tannakian
category over $\mathbb{Q}{}$ with fundamental group $P$, the Weil-number
protorus. For each $l$ (including $l=p$), $l$-adic cohomology defines a fibre
functor $\omega_{l}$ on $\Mot(\mathbb{F}{})$.
\end{theorem}

We recall the definition of the Weil-number torus $P$. An algebraic number
$\pi$ is said to be a \emph{Weil }$p^{n}$\emph{-number of weight }$m$ if, for
every embedding $\sigma\colon\mathbb{Q}{}[\pi]\rightarrow\mathbb{C}{}$,
$|\sigma\pi|=\left(  p^{n}\right)  ^{m/2}$ and, for some $N$, $p^{N}\pi$ is an
algebraic integer. Let $W(p^{n})$ be the set of Weil $p^{n}$-numbers (of any
weight) in $\mathbb{Q}{}^{\mathrm{al}}$. Then $W(p^{n})$ is a commutative
group, stable under the action of $\Gal(\mathbb{Q}{}^{\mathrm{al}}%
/\mathbb{Q}{})$. Let $W(p^{\infty})=\varinjlim_{n}W(p^{n})$. It is a torsion
free commutative group with an action of $\Gal(\mathbb{Q}{}^{\mathrm{al}%
}/\mathbb{Q}{})$, and $P$ is defined to be the protorus over $\mathbb{Q}{}$
with character group $X^{\ast}(P)=W(p^{\infty})$.

\begin{corollary}
\label{r26}If the Frobenius maps of the varieties in $\mathcal{S}{}$ are
semisimple, then for any theory of rational Tate classes on a class
$\mathcal{S}{}$, the functor%
\[
e\cdot hX(m)\mapsto e\cdot hX(m)\colon\Mot(\mathbb{F}{};\mathcal{S}{}%
_{0})\rightarrow\Mot(\mathbb{F}{};\mathcal{S}{})
\]
is an equivalence of tensor categories.
\end{corollary}

\begin{proof}
It is an exact tensor functor of tannakian categories over $\mathbb{Q}{}$ that
induces an isomorphism on the fundamental groups.
\end{proof}

\begin{corollary}
\label{r27}If the Frobenius maps of the varieties in $\mathcal{S}{}$ are
semisimple, a theory of rational Tate classes on $\mathcal{S}{}$ is determined
by its values on the objects in $\mathcal{S}{}_{0}$.
\end{corollary}

\begin{proof}
Let $X\in\mathcal{S}{}$, and choose an isomorphism $x\rightarrow h^{2r}X$ with
$x$ in $\Mot(\mathbb{F}{},\mathcal{S}{}_{0})$. The isomorphism $\omega
_{\mathbb{A}{}}(x)(r)\rightarrow\omega_{\mathbb{A}{}}(h^{2r}X)(r)\overset
{\textup{{\tiny def}}}{=}H_{\mathbb{A}{}}^{2r}(X)(r)$ maps
\[
\Hom(\1,x(r))\subset\Hom(\mathbb{A},\omega_{\mathbb{A}{}}(x)(r))=\omega
_{\mathbb{A}{}}(x)(r)
\]
onto $\mathcal{R}{}^{r}(X)$.
\end{proof}

\subsubsection{The category of motives as a quotient category}

In this subsubsection, we assume that the Frobenius maps are semisimple for
the varieties in $\mathcal{S}{}$. 

Let $\LMot(\mathbb{F})$ be the category of motives based on $\mathcal{S}{}%
_{0}$ using the Lefschetz classes as correspondences. It is a semisimple
tannakian category over $\mathbb{Q}{}$ (\cite{milne1999lm}). There is a
natural action of $P$ on the objects of $\LMot(\mathbb{F})$.

\begin{proposition}
\label{r25}For any theory of rational Tate classes on $\mathcal{S}{}$, the
natural functor
\[
q\colon\LMot(\mathbb{F})\rightarrow\Mot(\mathbb{F}{}),\quad e\cdot
hX(m)\rightsquigarrow e\cdot hX(m),
\]
is a quotient functor, and
\begin{equation}
\LMot(\mathbb{F})^{q}=\LMot(\mathbb{F})^{P}.\label{e12}%
\end{equation}
Conversely, every quotient functor $q\colon\LMot(\mathbb{F})\rightarrow
\mathsf{M}$ satisfying (\ref{e12}) and such that each standard fibre functor
factors through $q$ arises from a unique theory of rational Tate classes on
$\mathcal{S}{}_{0}$.
\end{proposition}

\begin{proof}
The first statement is obvious. Conversely, for each $x$ and $y$ in
$\LMot(\mathbb{F})$, the map%
\[
\Hom(x,y)\otimes_{\mathbb{Q}{}}\mathbb{Q}{}_{l}\overset{\omega_{l}%
}{\longrightarrow}\Hom(\omega_{l}(x),\omega_{l}(y))
\]
is injective (\cite{deligne1990}, 2.13). In particular, for each
$X\in\mathcal{S}_{0}{}$, $\omega_{l}$ defines an inclusion%
\[
\Hom(\1,h^{2r}X(r))\hookrightarrow\Hom(\mathbb{Q}{}_{l},H_{l}^{2r}%
(X)(r))\simeq H_{l}^{2r}(X)(r)
\]
for $l\neq p$, and similarly for $p$. On combining these maps, we get an
inclusion%
\[
\Hom(\1,h^{2r}X(r))\hookrightarrow H_{\mathbb{A}{}}^{2r}(X)(r)
\]
for each $X\in\mathcal{S}{}_{0}$, and we define $\mathcal{R}{}^{r}(X)$ to be
the image of this map. The family $(\mathcal{R}{}^{\ast}(X))_{X\in
\mathcal{S}{}_{0}}$ with $\mathcal{R}{}^{\ast}(X)=\bigoplus\nolimits_{r}%
\mathcal{R}{}^{r}(X)$ satisfies (R1,R2), and (\ref{e12}) implies that it
satisfies (R4).
\end{proof}

\begin{corollary}
\label{r28}To give a theory of rational Tate classes on $\mathcal{S}{}_{0}$ is
the same as to give a $\mathbb{Q}{}$-structure on the restriction of
$\omega_{\mathbb{A}{}}$ to $\LMot(\mathbb{F})^{P}$, i.e., a subfunctor
$\omega_{0}\subset\omega_{\mathbb{A}{}}$ such that $\mathbb{A}\otimes
_{\mathbb{Q}{}}{}\omega_{0}(x)\simeq\omega_{\mathbb{A}{}}(x)$ for all
$x\in\LMot(\mathbb{F})^{P}$.
\end{corollary}

\begin{proof}
Obvious from the above.
\end{proof}

\section{Good theories of rational Tate classes}

In this section, ${}\mathcal{S}{}$ consists of the varieties over
$\mathbb{F}{}$ whose Frobenius elements are semisimple. Clearly $\mathcal{S}%
{}$ satisfies the condition (*), and includes $\mathcal{S}{}_{0}$ (by a
theorem of Weil). Conjecturally, $\mathcal{S}{}$ includes all varieties over
$\mathbb{F}{}$.

\subsection{Definition}

An abelian variety with sufficiently many endomorphisms over an algebraically
closed field of characteristic zero will be called a \emph{CM abelian
variety}. Let $\mathbb{Q}{}^{\mathrm{al}}$ be the algebraic closure of
$\mathbb{Q}{}$ in $\mathbb{C}{}$. The functor $A\rightsquigarrow
A_{\mathbb{C}{}}$ from CM abelian varieties over $\mathbb{Q}{}^{\mathrm{al}}$
to CM abelian varieties over $\mathbb{C}{}$ is an equivalence of categories
(see, for example, \cite{milneCM}, \S 7).

Fix a $p$-adic prime $w$ of $\mathbb{Q}{}^{\mathrm{al}}$, and let
$\mathbb{F}{}$ be its residue field. Thus, $\mathbb{F}{}$ is an algebraic
closure of $\mathbb{F}{}_{p}$. It follows from the theory of N\'{e}ron models
that there is a well-defined reduction functor $A\rightsquigarrow A_{0}$
sending a CM abelian variety over $\mathbb{Q}{}^{\mathrm{al}}$ to an abelian
variety over $\mathbb{F}{}$ (\cite{serreT1968}, Theorem 6).

For a variety $X$ over an algebraically closed field of characteristic zero,
we write%
\[
H_{\mathbb{A}{}}^{\ast}(X)=\left(  \varprojlim_{m}H^{\ast}(X_{\mathrm{et}%
},\mathbb{Z}{}/m\mathbb{Z}{})\otimes_{\mathbb{Z}{}}\mathbb{Q}{}\right)  \times
H_{\mathrm{dR}}^{\ast}(X)\text{,}%
\]
and for a variety $X_{0}$ over $\mathbb{F}{}$, we now write%
\[
H_{\mathbb{A}}^{\ast}(X_{0})=\left(  \varprojlim_{p\nmid m}H^{\ast
}(X_{0\mathrm{et}},\mathbb{Z}{}/m\mathbb{Z}{})\otimes_{\mathbb{Z}{}}%
\mathbb{Q}{}\right)  \times H_{p}^{\ast}(X_{0})\otimes_{B(\mathbb{F)}%
}\mathbb{Q}{}_{w}^{\mathrm{al}},
\]
where $\mathbb{Q}{}_{w}^{\mathrm{al}}$ is the completion of $\mathbb{Q}%
{}^{\mathrm{al}}$ at $w$. If $X$ has good reduction to $X_{0}$ at $w$, then%
\begin{align*}
H^{\ast}(X_{\mathrm{et}},\mathbb{Z}{}/m\mathbb{Z}{})  &  \simeq H^{\ast
}(X_{0\mathrm{et}},\mathbb{Z}{}/m\mathbb{Z}{})\text{ for all }m\text{ not
divisible by }p\text{, and}\\
H_{\mathrm{dR}}^{\ast}(X)\otimes_{\mathbb{Q}{}^{\mathrm{al}}}\mathbb{Q}{}%
_{w}^{\mathrm{al}}  &  \simeq H_{p}^{\ast}(X_{0})\otimes_{B(\mathbb{F}{}%
)}\mathbb{Q}{}_{w}^{\mathrm{al}},
\end{align*}
and so there is a canonical map $H_{\mathbb{A}{}}^{\ast}(X)\rightarrow
H_{\mathbb{A}{}}^{\ast}(X_{0})$, called the \emph{specialization map}.

For a variety $X$ over a field of characteristic zero, $\mathcal{B}{}^{\ast
}(X)$ denotes the $\mathbb{Q}{}$-subalgebra of absolute Hodge classes in
$H_{\mathbb{A}{}}^{2\ast}(X)(\ast)$. Because of Deligne's theorem
(\citeyear{deligne1982}), I refer to the absolute Hodge classes on a variety
$X\in\mathcal{S}{}_{0}$ simply as Hodge classes.

\begin{definition}
\label{r29}A theory of rational Tate classes $\mathcal{R}{}$ on $\mathcal{S}%
{}$ (over $\mathbb{F}{}$) is \emph{good} if

\begin{enumerate}
\item[(R3)] for all CM abelian varieties $A$ over $\mathbb{Q}{}^{\mathrm{al}}%
$, the Hodge classes on $A$ map to elements of $\mathcal{R}{}^{\ast}(A_{0})$
under the specialization map $H_{\mathbb{A}{}}^{2\ast}(A)(\ast)\rightarrow
H_{\mathbb{A}{}}^{2\ast}(A_{0})(\ast)$.
\end{enumerate}
\end{definition}

\noindent In other words, (R3) requires that there exists a commutative diagam%
\[
\renewcommand{\arraystretch}{1.3}%
\begin{array}
[c]{ccc}%
\mathcal{B}{}^{\ast}(A) & \subset & H_{\mathbb{A}{}}^{2\ast}(A)(\ast)\\
\downarrow &  & \phantom{ab}\downarrow\scriptstyle{\text{specialization}}\\
\mathcal{R}{}^{\ast}(A) & \subset & H_{\mathbb{A}{}}^{2\ast}(A_{0})(\ast).
\end{array}
\]
Recall (\cite{deligne1982}, 2.9b) that $\Gal(\mathbb{Q}{}^{\mathrm{al}%
}/\mathbb{Q}{})$ acts on $\mathcal{B}{}^{\ast}(A)$ through a finite quotient,
and so the Hodge classes on $A$ are Tate classes. Therefore, they specialize
to Tate classes on $A_{0}$, i.e., there is a commutative diagram%
\begin{equation}
\renewcommand{\arraystretch}{1.3}%
\begin{array}
[c]{ccc}%
\mathcal{B}{}^{\ast}(A) & \subset & H_{l{}}^{2\ast}(A)(\ast)\\
\downarrow &  & \phantom{\scriptstyle \simeq}\downarrow{\scriptstyle\simeq}\\
\mathcal{T}_{l}^{\ast}(A) & \subset & H_{l{}}^{2\ast}(A_{0})(\ast)
\end{array}
\label{e19}%
\end{equation}
for each $l$, except that for $l=p$ the cohomology groups have to be tensored
with $\mathbb{Q}{}_{w}^{\mathrm{al}}$.

\subsection{The fundamental theorems}

\begin{theorem}
\label{r21}A family $(\mathcal{R}{}^{\ast}(X))_{X\in\mathcal{S}{}_{0}}$ is a
good theory of rational Tate classes on $\mathcal{S}{}_{0}$ if it satisfies
the conditions (R1), (R2), and (R3), and the following weakening of (R4):

\begin{enumerate}
\item[(R4*)] for all varieties $X$ in $\mathcal{S}{}$, the $\mathbb{Q}{}%
$-algebra $\mathcal{R}{}^{\ast}(X)$ is of finite degree, and for all primes
$l$, the projection map $H_{\mathbb{A}{}}^{2\ast}(X)(\ast)\rightarrow
H_{l}^{2\ast}(X)(\ast)$ sends $\mathcal{R}{}^{\ast}(X)$ into $\mathcal{T}%
{}_{l}^{\ast}(X)$.
\end{enumerate}
\end{theorem}

\noindent In other words, instead of requiring $\mathcal{R}{}^{\ast}(X)$ to be
a $\mathbb{Q}{}$-structure on $\mathcal{T}_{l}^{\ast}(X)$ for all $l$, we
merely require that it be finite dimensional and map into $\mathcal{T}{}%
_{l}^{\ast}(X)$ for all $l$.

\begin{proof}
We fix a CM-subfield $K$ of $\mathbb{C}{}$ that is finite and Galois over
$\mathbb{Q}{}$ and contains a quadratic imaginary number field in which $p$
splits, and we let $\Gamma=\Gal(K/\mathbb{Q}{})$. Let $\ell$ be a prime $\neq
p$, and let $A$ be an abelian variety over $\mathbb{Q}{}^{\mathrm{al}}$ split
by $K$ (i.e., such that $\End^{0}(A)$ is split by $K$).

The inclusion $\End^{0}(A)\hookrightarrow\End^{0}(A_{0})$ maps the centre
$C(A)$ of $\End^{0}(A)$ onto a $\mathbb{Q}{}$-subalgebra of $\End^{0}(A_{0})$
containing its centre $C(A_{0})$, and hence it defines an inclusion
$L(A_{0})\rightarrow L(A)$ of Lefschetz groups. Consider the diagram%
\[
\xymatrix{
\MT(A)\ar@{^{(}->}[r]&L(A)\\
P(A_{0})\ar@{-->}[u]\ar@{^{(}->}[r]&L(A_{0})\ar@{^{(}->}[u]
}
\]
in which $MT(A)$ is the Mumford-Tate group of $A$ and $P(A_{0})$ is the
smallest algebraic subgroup of $L(A_{0})$ containing a Frobenius endomorphism
of $A_{0}$. Almost by definition, $\MT(A)$ is the largest algebraic subgroup
of $L(A)$ fixing the Hodge classes in $H_{B}^{2\ast}(A_{\mathbb{C}{}}%
^{n})(\ast)$ for all $n$, and so $\MT(A)_{\mathbb{Q}{}_{\ell}}$ is the largest
algebraic subgroup of $L(A)_{\mathbb{Q}{}_{\ell}}$ fixing the Hodge classes in
$H_{\ell}^{2\ast}(A^{n})(\ast)$ for all $n$. On the other hand, the classes in
$H_{\ell}^{2\ast}(A_{0}^{n})(\ast)$ fixed by $P(A_{0})_{\mathbb{Q}{}_{l}}$ are
exactly the Tate classes. The specialization map $H_{\ell}^{2\ast}%
(A)(\ast)\rightarrow H_{\ell}^{2\ast}(A_{0})(\ast)$ is equivariant for the
homomorphism $L(A_{0})\rightarrow L(A)$. From (\ref{e19}), we see that
$P\left(  A_{0}\right)  _{\mathbb{Q}{}_{\ell}}\subset\MT(A)_{\mathbb{Q}%
{}_{\ell}}$ (inside $L(A)_{\mathbb{Q}{}_{\ell}}$). This implies that
$P(A_{0})\subset\MT(A)$ (inside $L(A)$), and explains the left hand arrow in
the above diagram.

Now choose $A$ to be so large that every simple abelian variety over
$\mathbb{Q}{}^{\mathrm{al}}$ split by $K$ is isogenous to an abelian
subvariety of $A$. Then $A_{0}$ is an abelian variety over $\mathbb{F}{}$ such
that every abelian variety over $\mathbb{F}{}$ split by $K$ is isogenous to an
abelian subvariety of $A_{0}$ (see \cite{milne2007aim}, 8.7). With this choice
of $A$, the groups $L(A)$, $L(A_{0})$, $\MT(A)$, and $P(A_{0})$ are equal to
the groups denoted $T^{K}$, $L^{K}$, $S^{K}$, and $P^{K}$ in
\cite{milne1999lm}, and so (ibid., Theorem 6.1),%
\begin{equation}
P(A_{0})=\MT(A)\cap L(A_{0})\quad\text{(intersection inside }L(A)\text{).}
\label{e9}%
\end{equation}

Let $\mathcal{R}_{\ell}^{\ast}(A)$ be the image of $\mathcal{R}{}^{\ast}(A)$
in $H_{\ell}^{2\ast}(A)(\ast)$. The hypotheses of Theorem \ref{r3} hold with
$H_{W}=H_{\ell}$ for an infinite set of primes $\ell$ (see the proof of
Theorem \ref{r5}). In particular, there exists a prime $\ell$ such that the
product pairings
\[
\mathcal{R}{}_{\ell}^{r}(A)\times\mathcal{R}{}_{\ell}^{\dim A-r}%
(A)\rightarrow\mathcal{R}{}^{\dim A}(A)\simeq\mathbb{Q}{}%
\]
are nondegenerate for all $r$. Let $G$ be the largest algebraic subgroup of
$L(A_{0})_{\mathbb{Q}{}_{\ell}}$ fixing the rational Tate classes in $H_{\ell
}^{2\ast}(A_{0}^{n})(\ast)$ for all $n$. The group $G$ acts on $H_{\ell
}^{2\ast}(A^{n})(\ast)$ through the homomorphisms $G\rightarrow L(A_{0}%
)_{\mathbb{Q}{}_{\ell}}\rightarrow L(A)_{\mathbb{Q}{}_{\ell}}$, and it fixes
the Hodge classes (because of (R3)). Therefore,
\[
G\subset\MT(A)_{\mathbb{Q}{}_{\ell}}\cap L(A_{0})_{\mathbb{Q}{}_{\ell}%
}=P(A_{0})_{\mathbb{Q}{}_{\ell}},
\]
and so $G$ fixes all Tate classes in $H_{\ell}^{2\ast}(A_{0}^{n})(\ast)$ (all
$n$). According to (\ref{r6c}), this implies that the space of Tate classes in
$H_{\ell}^{2\ast}(A_{0}^{n})(\ast)$ (all $n$) is spanned by $\mathcal{R}%
{}^{\ast}(A_{0})$. Because the Frobenius maps on abelian varieties are
semisimple (Weil's theorem), Theorem \ref{r6} shows that the maps
$\mathcal{R}{}^{\ast}(A_{0}^{n})\otimes_{\mathbb{Q}{}}\mathbb{Q}{}%
_{l}\rightarrow\mathcal{T}{}_{l}^{\ast}(A_{0})$ are isomorphisms for all $l$
(including $l=p$). It follows that the same is true of every abelian
subvariety of some power $A_{0}$ (because it is an isogeny factor), i.e., for
all abelian varieties over $\mathbb{F}{}$ split by $K$. Since every abelian
variety over $\mathbb{F}{}$ is split by some CM-field, this completes the proof.
\end{proof}

\begin{theorem}
\label{r22}There exists at most one good theory of rational Tate classes on
$\mathcal{S}{}{}$.
\end{theorem}

\begin{proof}
It suffices to prove this with $\mathcal{S}{}=\mathcal{S}{}_{0}$ (see
\ref{r27}). Certainly, if $\mathcal{R}{}_{1}^{\ast}$ and $\mathcal{R}{}%
_{2}^{\ast}$ are two theories of rational Tate classes and one is contained in
the other, then they are equal (by condition (R4)). But Theorem \ref{r9} shows
that if $\mathcal{R}{}_{1}^{\ast}$ and $\mathcal{R}{}_{2}^{\ast}$ are
\textit{good} theories of rational Tate classes on $\mathcal{S}{}%
_{0}(\mathbb{F}{})$, then $\mathcal{R}_{1}^{\ast}\cap\mathcal{R}{}_{2}^{\ast}$
is also a good theory of rational Tate classes, and so it is equal to each of
$\mathcal{R}{}_{1}^{\ast}$ and $\mathcal{R}{}_{2}^{\ast}$.
\end{proof}

\begin{theorem}
[Milne \textnf{\citeyear{milne1999lm}}]\label{r23a}If the Hodge classes on CM
abelian varieties over $\mathbb{Q}{}^{\mathrm{al}}$ specialize to algebraic
classes on abelian varieties over $\mathbb{F}$, then the Tate conjecture holds
for abelian varieties over $\mathbb{F}$. In particular, the Hodge conjecture
for CM abelian varieties over $\mathbb{Q}{}^{\mathrm{al}}$ implies the Tate
conjecture for abelian varieties over $\mathbb{F}{}$.
\end{theorem}

\begin{proof}
Let $\mathcal{A}{}^{\ast}(X)$ be the $\mathbb{Q}{}$-subalgebra of
$H_{\mathbb{A}{}}^{\ast}(X)$ generated by the algebraic classes. Theorem
\ref{r21} shows that $\mathcal{A}{}^{\ast}$ is a good theory of rational Tate
classes on $\mathcal{S}{}_{0}$.
\end{proof}

\begin{remark}
\label{r23b}For any CM subfield $K$ of $\mathbb{C}{}$ finite and Galois over
$\mathbb{Q}{}$, Hazama (2002, 2003)\nocite{hazama2002, hazama2003} constructs
a CM abelian variety $A$ with the following properties:

\begin{itemize}
\item $A$ is split by $K$ and every simple CM abelian variety split by $K$ is
isogenous to an abelian subvariety of $A$, and

\item for all $n\geq0$, the $\mathbb{Q}{}$-algebra of Hodge classes on $A^{n}$
is generated by those of degree $\leq2$.
\end{itemize}

\noindent It follows that, in order to prove the Hodge conjecture for CM
abelian varieties, it suffices to prove it in codimension $2$. On combining
Hazama's ideas with those from \cite{milne1999lm}, one can show that in order
to prove the Tate conjecture for abelian varieties over $\mathbb{F}{}$, it
suffices to prove it in codimension $2$ (\cite{milne2007aim}, 8.6).
\end{remark}

\subsection{Motives defined by a good theory of rational Tate classes}

Recall (\ref{r20}) that a theory of rational Tate classes $\mathcal{R}{}%
^{\ast}$ on $\mathcal{S}{}$ defines a semisimple tannakian category of motives
$\Mot(\mathbb{F}{})$ with fundamental group $P$. Moreover (\ref{r25}), there
is a quotient functor $\LMot(\mathbb{F})\rightarrow\Mot(\mathbb{F}{})$ bound
by a homomorphism of fundamental group $P\rightarrow L$. When $\mathcal{R{}%
}^{\ast}$ is a \emph{good} theory, then this extends to a commutative diagram
of exact tensor functors of semisimple tannakian categories, as at left, bound
by the commutative diagram of fundamental groups at right:%
\begin{equation}
\begin{CD} \mathsf{CM}(\mathbb{Q}^{\mathrm{al}})@<J<<\LCM(\mathbb{Q}^{\text{al}})\\ @VV{R}V@VV{R^L}V\\ \Mot(\mathbb{F})@<I<<\LMot(\mathbb{F}) \end{CD}\hspace
{1in}\quad\begin{CD} S@>>>T\\ @AAA@AAA\\ P@>>>L. \end{CD} \label{e18}%
\end{equation}
Here:

\begin{itemize}
\item $\CM(\mathbb{Q}^{\text{al}})$ is the category of motives based on the CM
abelian varieties over $\mathbb{Q}{}^{\mathrm{al}}$ using the Hodge classes as
correspondences. Its fundamental group is the Serre group $S$.

\item $\LCM(\mathbb{Q}{}^{\text{al}})$ is the similar category, except using
the Lefschetz classes as correspondences. Its fundamental group is a certain
pro-algebraic group $T$ of multiplicative type.

\item The horizontal functors are of the form $e\cdot hX(r)\rightsquigarrow
e\cdot hX(r)$, and the vertical functors are of the form $e\cdot
hX(r)\rightsquigarrow e\cdot hX_{0}(r)$.

\item The groups and homomorphisms in the diagram at right have elementary
explicit descriptions, and the homomorphisms are all injective.

\item For each $l$ (including $l=p$), there exists a fibre functor $\omega
_{l}$ on $\Mot(\mathbb{F}{})$ such that $\omega_{l}\circ R$ and $\omega
_{l}\circ I$ are equal (meaning \textit{really} equal) to the standard fibre functors.
\end{itemize}

\noindent See \cite{milne1999lm}.

The last statement places a condition on $\Mot(\mathbb{F}{})$ for every
\textit{finite} prime. We shall also need a condition at the infinite prime,
and this is expressed in terms of polarizations on Tate triples (see
\cite{deligneM1982}, \S 5, for this theory).

A divisor $D$ on an abelian variety $A$ over $\mathbb{F}{}$ defines a pairing
$\psi_{D}\colon h_{1}A\times h_{1}A\rightarrow\mathbb{T}{}$, which is a Weil
form if $D$ is very ample (\cite{weil1948}, Th\'{e}or\`{e}me 38). A Weil form
arising in this way from a very ample divisor is said to be \emph{geometric.}

The categories in (\ref{e18}) all have natural Tate triple structures which
are preserved by the functors. Moreover, each of the categories
$\CM(\mathbb{Q}^{\text{al}})$, $\LCM(\mathbb{Q}{}^{\text{al}})$, and
$\LMot(\mathbb{F})$ has a unique polarization $\Pi^{\CM}$, $\Pi^{\LCM}$,
$\Pi^{\LMot}$ called the \emph{geometric polarization}, for which the
geometric Weil forms are positive. More precisely, for each homogeneous object
$X$ in the category, the geometric Weil forms on $X$ are contained in a single
equivalence class $\Pi(X)$, and the family $(\Pi(X))_{X}$ is a polarization on
the Tate triple. Moreover $J\colon\Pi^{\LCM}\mapsto\Pi^{\CM}$ and $R^{L}%
\colon\Pi^{\LCM}\mapsto\Pi^{\LMot}$. See \cite{milne2002}\nocite{milne2002p},
1.1, 1.5.

\begin{lemma}
\label{r30}Let $\mathcal{S}{}=\mathcal{S}{}_{0}$. There exists a unique
polarization $\Pi$ on $\Mot(\mathbb{F}{})$ such that $R\colon\Pi^{\CM}%
\mapsto\Pi$.
\end{lemma}

\begin{pf}
[following \textnf{\cite{milne2002p}}, proof of 2.1]Fix a CM subfield $K$ of
$\mathbb{C}{}$ such that $K$ is finite and Galois over $\mathbb{Q}{}$ and $K$
properly contains an imaginary quadratic field in which $p$ splits. Let
$\CM^{K}(\mathbb{Q}{}^{\mathrm{al}})$ and $\Mot^{K}(\mathbb{F}{})$ denote the
tannakian subcategories of $\CM(\mathbb{Q}{}^{\mathrm{al}})$ and
$\Mot(\mathbb{F}{})$ generated by the abelian varieties split by $K$. It
suffices to prove the proposition for $R^{K}\colon\CM^{K}(\mathbb{Q}%
{}^{\mathrm{al}})\rightarrow\Mot^{K}(\mathbb{F}{})$.

Let $A$ be a CM abelian variety over $\mathbb{Q}{}^{\mathrm{al}}$ split by $K$
such that every simple CM abelian variety over $\mathbb{Q}{}^{\mathrm{al}}$
split by $K$ is isogenous to a subvariety of $A$, and let $X=\underline
{\End}(h_{1}A)^{P}$. It follows from \cite{milne1999lm} that $S^{K}/P^{K}$
acts faithfully on $X$,\footnote{As Yves Andr\'{e} pointed out to me, this is
not entirely obvious, so I include a proof. I begin with an elementary remark.
Let $T\supset L$ be tori with $T$ acting on a finite dimensional vector space
$V$. Let $\chi_{1},\ldots,\chi_{n}$ be the characters of $T$ occurring in $V$.
Then $T$ acts faithfully on $V$ if and only if $\chi_{1},\ldots,\chi_{n}$ span
$X^{\ast}(T)$ as a $\mathbb{Z}{}$-module --- assume this. The characters of
$T$ occurring in $\End(V)$ are $\{\chi_{i}-\chi_{j}\}$, and the set of those
occurring in $\End(V)^{L}$ is%
\begin{equation}
\{\chi_{i}-\chi_{j}\mid\chi_{i}|L=\chi_{j}|L\}. \tag{*}%
\end{equation}
On the other hand,%
\begin{equation}
X^{\ast}(T/L)=\{{\textstyle\sum}a_{i}\chi_{i}\mid{\textstyle\sum}a_{i}\chi
_{i}|L=0\}. \tag{**}%
\end{equation}
Thus, $T/L$ will act faithfully on $\End(V)^{L}$ if the set (*) spans the
$\mathbb{Z}{}$-module (**).
\par
I now prove the statement. With the notations of \cite{milne1999lm}, \S 6
(especially p69), $T^{\Psi}$ acts on a realization of $h_{1}A^{\Psi}$ through
the characters $\psi_{0},\ldots,\psi_{n-1},\iota\psi_{0},\ldots,\iota
\psi_{n-1}$, where the $\psi_{i}$ have been numbered so that $\pi(\psi
_{0})=\cdots=\pi(\psi_{d-1})=\pi_{0}$, $\pi(\psi_{d})=\cdots=\pi(\psi
_{2d-1})=\pi_{1},$ etc.. Now $\sum a_{i}\cdot\psi_{i}|L^{\Psi}={\textstyle\sum
}a_{i}\cdot\pi(\psi_{i})$, which is zero if and only if ${\textstyle\sum
_{i=0}^{d-1}}a_{i}=0$, ${\textstyle\sum_{i=d}^{2d-1}}a_{i}=0$, \ldots\ ; but
then ${\textstyle\sum}a_{i}\psi_{i}={\textstyle\sum_{i=0}^{d-1}}a_{i}(\psi
_{i}-\psi_{0})+\cdots$, which (by the remark) shows that $T^{\Psi}/L^{\Pi}$
acts faithfully on $\underline{\End}(h_{1}A^{\Psi})^{L^{\Pi}}$. Similarly
$T^{\bar{\Psi}}/L^{\bar{\Pi}}$ acts faithfully on $\underline{\End}%
(h_{1}A^{\bar{\Psi}})^{L^{\bar{\Pi}}}$ and it follows that $T^{A^{\Psi}\times
A^{\bar{\Psi}}}/L^{A^{\Pi}\times A^{\bar{\Pi}}}$ acts faithfully on
$\underline{\End}(h_{1}(A^{\Psi}\times A^{\bar{\Psi}}))^{L^{A^{\Pi}\times
A^{\bar{\Pi}}}}$. As $P^{K}/L^{K}\hookrightarrow T^{A^{\Psi}\times
A^{\bar{\Psi}}}/L^{A^{\Pi}\times A^{\bar{\Pi}}}$ (cf. ibid. Lemma 6.9), this
implies the statement.} and hence that $X$ generates $\Mot^{K}$.

Let $\phi$ be the geometric Weil form on $h_{1}A$ defined by an ample divisor
$D$ on $A$, and let $\psi=T^{\phi}|X$, where $T^{\phi}$ is the symmetric
bilinear form on $\underline{\End}(X)$ defined by $\phi$ (\cite{milne2002p},
1.1). Then $\psi\in\Pi^{\CM}(X)$, and it suffices to show that $R^{K}(\psi)$
is positive-definite (ibid. 1.4, 1.5). But $R^{K}(X)=\End^{0}(A_{0})$ and
$R^{K}(\psi)$ is the trace pairing $u,v\mapsto\Tr(u\cdot v^{\dagger})$ of the
Rosati involution defined by $D_{0}$ on $A_{0}$, which is positive definite by
Th\'{e}or\`{e}me 38 of \cite{weil1948}.
\end{pf}

\begin{theorem}
\label{r31}There exists a unique polarization $\Pi$ on $\Mot(\mathbb{F}{})$
such that

\begin{enumerate}
\item the geometric Weil forms are positive,

\item $R\colon\Pi^{\CM}\mapsto\Pi$, and

\item $I\colon\Pi^{\LMot}\mapsto\Pi$.
\end{enumerate}

\noindent Moreover, each of these conditions determines $\Pi$ uniquely.
\end{theorem}

\begin{proof}
The uniqueness being obvious, it remains to prove the existence. As
$\Mot(\mathbb{F}{};\mathcal{S}{}_{0})\rightarrow\Mot(\mathbb{F}{}%
;\mathcal{S})$ is an equivalence, there exists an unique polarization $\Pi$ on
$\Mot(\mathbb{F}{};\mathcal{S}{})$ such that $R\colon\Pi^{\CM}\mapsto\Pi$. The
geometric Weil forms are positive for $\Pi^{\CM}$, and every polarized abelian
variety over $\mathbb{F}{}$ is isogenous to the reduction of a polarized CM
abelian variety over $\mathbb{Q}^{\mathrm{al}}$ (\cite{zink1983}, 2.7), and so
if $R\colon\Pi^{\CM}\rightarrow\Pi$, then every geometric Weil form on a
homogeneous factor of the motive of an abelian variety is positive, but all
homogeneous objects in $\Mot(\mathbb{F}{})$ are such factors. This proves that
$\Pi$ has the properties (a) and (b), and property (c) follows obviously from (a).
\end{proof}

\begin{aside}
\label{r32}In fact, $\Pi$ is the only polarization on $\Mot(\mathbb{F}{})$ for
which the geometric Weil forms on a supersingular elliptic curve are positive
(\cite{milne1994}, 3.17c).
\end{aside}

\subsection{The Hodge standard conjecture}

Let $\mathcal{R}{}^{\ast}$ be a good theory of rational Tate classes on
$\mathcal{S}{}$, and fix a prime $l$. Let $X\in\mathcal{S}{}$, and let
$L\colon H_{l}^{r}(X)\rightarrow H_{l}^{r+2}(X)(1)$ be the Lefschetz operator
defined by a smooth hyperplane section of $X$. When $X$ is connected, the
\emph{primitive part } of $\mathcal{R}{}^{r}(X)$ is defined to be%
\[
\mathcal{R}^{r}(X)_{\mathrm{prim}}=\{z\in\mathcal{R}{}^{r}(X)\mid L^{\dim
X-2r+1}z=0\}.
\]
The next theorem shows that the Hodge standard conjecture holds for rational
Tate classes.

\begin{theorem}
\label{r33}For every connected $X\in\mathcal{S}{}$ and $r\leq\frac{1}{2}\dim
X$, the bilinear form $\theta^{r}$%
\begin{equation}
x,y\mapsto(-1)^{r}\langle L^{\dim X-2r}x\cdot y\rangle\colon\mathcal{R}{}%
^{r}(X)_{\mathrm{prim}}\times\mathcal{R}{}^{\dim X-r}(X)_{\mathrm{prim}%
}\rightarrow\mathcal{R}{}^{\dim X}(X)\simeq\mathbb{Q}{} \label{e24}%
\end{equation}
is positive definite.
\end{theorem}

Let $d=\dim X$. Let $p^{r}(X)$ be the largest subobject of%
\[
\Ker(L^{d-2r+1}\colon h^{2r}(X)(r)\rightarrow h^{2d-2r+2}(X)(d-r+1)
\]
on which $\pi\overset{\text{{\tiny def}}}{=}\pi(\Mot(\mathbb{F}{}))$ acts
trivially. Then%
\[
\Hom(\1,p^{r}(X))=\mathcal{R}{}^{r}(X)_{\mathrm{prim}}%
\]
and there is a pairing%
\[
\vartheta^{r}\colon p^{r}(X)\otimes p^{r}(X)\rightarrow\1,
\]
also fixed by $\pi$, such that $\Hom(\1,\vartheta^{r})=\theta^{r}$. Theorem
\ref{r33} follows from Theorem \ref{r31} and the next two lemmas.

\begin{lemma}
\label{r34}If $\Mot(\mathbb{F}{})$ admits a polarization for which the forms
$\vartheta^{r}$ are positive, then the pairings $\theta^{r}$ are positive definite.
\end{lemma}

\begin{proof}
See the proof of \cite{milne2002p}, 4.5.
\end{proof}

\begin{lemma}
\label{r35}If $\Pi$ is a polarization of $\Mot(\mathbb{F}{})$ for which
$R\colon\Pi^{\CM}\mapsto\Pi$, then the forms $\vartheta^{r}$ are positive for
$\Pi$.
\end{lemma}

\begin{pf}
[following Milne \textnf{\citeyear{milne2002p}}, 4.5)]Let $A_{1}$ be a
polarized abelian variety over $\mathbb{F}{}$. According to \cite{zink1983},
there exists an abelian variety $A$ over $\mathbb{Q}{}^{\mathrm{al}}$ and an
isogeny $A_{0}\rightarrow A_{1}$. The bilinear forms%
\[
\varphi^{r}\colon h^{r}A\otimes h^{r}A\overset{\id\otimes\ast}{\longrightarrow
}h^{r}A\otimes h^{2d-r}(A)(d-r)\rightarrow h^{2n}(A)(d-r)\simeq\1(-r)
\]
are positive for the polarization $\Pi^{\CM}$ (cf. \cite{saavedra1972}, VI
4.4) --- here $d=\dim A$ and $\ast$ is defined by the given polarization on
$A$. The restriction of $\varphi^{2r}\otimes\id_{\1(2r)}$ to the subobject
$p^{r}(A)$ of $h^{2r}(A)(r)$ is of the form $\vartheta^{r}$, which is
therefore positive (\cite{deligneM1982}, 4.11b). Because of the isogeny
$A_{0}\rightarrow A_{1}$ and our hypothesis on $\Pi$, the similar statement is
true for $A_{1}$. As every object of $\Mot(\mathbb{F}{})$ is a direct factor
of the motive of an abelian variety, this proves the result.
\end{pf}

\begin{corollary}
\label{r36}If there exists a good theory of rational Tate classes such that
all algebraic classes are rational Tate classes, then the Hodge standard
conjecture holds for all $X\in\mathcal{S}{}$.
\end{corollary}

\begin{proof}
The form (\ref{e24}) is positive definite if and only if the quadratic form
$x\mapsto\langle x\!\cdot\! \ast x\rangle$ on $\mathcal{R}{}^{r}%
(X)_{\text{\textrm{prim}}}$ is positive definite. The restriction of a
positive definite quadratic form to a subspace is positive definite.
\end{proof}

\begin{remark}
\label{r23c}Let $\mathcal{S}{}$ contain all varieties over an algebraically
closed field $k$, and let $H_{W}$ be a Weil cohomology theory with coefficient
field $Q$. Andr\'{e} (1996) \nocite{andre1996}defines a countable subfield
$Q_{0}$ of $Q$ and constructs a family $(\mathcal{\mathcal{R}{}}^{\ast
}(X))_{X\in\mathcal{S}{}}$ of $Q{}_{0}$-subalgebras of $H_{W}^{2\ast}%
(X)(\ast)$ that is the smallest containing the algebraic classes, the
Lefschetz operator $\Lambda$, and satisfying (R1) --- the elements of
$\mathcal{R}{}^{\ast}(X)$ are called the \emph{motivated classes} on $X$. When
$H_{W}$ is $\ell$-adic \'{e}tale cohomology with $\ell$ distinct from the
characteristic of $k$, he has proved the following:

\begin{enumerate}
\item the motivated classes on abelian varieties in characteristic zero are
exactly the Hodge classes (\cite{andre1996});

\item the motivated classes on a CM abelian variety over $\mathbb{Q}%
{}^{\mathrm{al}}$ specialize to motivated classes on $A_{0}$
(\cite{andre2006b}, 2.4.1).
\end{enumerate}

\noindent On applying the obvious variant of Theorem \ref{r21}, one finds that
the motivated classes on abelian varieties over $\mathbb{F}{}$ form a theory
of rational Tate classes in $H_{\ell}$ (in the sense on p\pageref{standard}),
except that $\mathbb{Q}{}$ must be replaced by $Q_{0}$. In particular, the
space of motivated classes in $H_{\ell}^{2\ast}(A_{0})(\ast)$ is a $Q_{0}%
$-structure on $\mathcal{T}^{\ast}(A_{0})$. If $Q_{0}$ is formally real, the
obvious variant of Theorem \ref{r33} implies the Hodge standard conjecture for
abelian varieties over $\mathbb{F}{}$.
\end{remark}

\subsection{Finite fields}

Suppose that we have a good theory of rational Tate classes $\mathcal{R}{}$ on
some class $\mathcal{S}{}$ of varieties over $\mathbb{F}{}$. For any variety
$X$ over a finite subfield $\mathbb{F}{}_{q}$ of $\mathbb{F}{}$ such that
$X_{\mathbb{F}{}}\in\mathcal{S}{}$, $\Gal(\mathbb{F}{}/\mathbb{F}{}_{q})$ acts
through a finite quotient on $\mathcal{R}{}^{\ast}(X_{\mathbb{F}{}})$ because
it acts continuously, and a countable profinite group is finite. In this case,
we define%
\[
\mathcal{R}{}^{\ast}(X)=\mathcal{R}{}^{\ast}(X_{\mathbb{F}{}}%
)^{\Gal(\mathbb{F}{}/\mathbb{F}_{q})}.
\]

\section{The rationality conjecture}

In this section, I state a conjecture that has many of the same consequences
for motives over $\mathbb{F}{}$ as the Hodge conjecture for CM abelian
varieties but appears to be much more accessible.

\subsection{Statement}

\begin{rc1}
\label{rc1}Let $A$ be an abelian variety over $\mathbb{Q}{}^{\mathrm{al}}$
with good reduction to an abelian variety $A_{0}$ over $\mathbb{F}{}$. The cup
product of the specialization to $A_{0}$ of any Hodge class on $A$ with any
Lefschetz class of complementary dimension lies in $\mathbb{Q}{}$.
\end{rc1}

In more detail, a Hodge class on $A$ is an element of $\gamma$ of
$H_{\mathbb{A}{}}^{2\ast}(A)(\ast)$ and its specialization $\gamma_{0}$ is an
element of $H_{\mathbb{A}{}}^{2\ast}(A_{0})(\ast)$. Thus the cup product
$\gamma_{0}\cup\delta$ of $\gamma_{0}$ with a Lefschetz class of complementary
dimension $\delta$ lies in
\[
H_{\mathbb{A}{}}^{2d}(A_{0})(d)\simeq\mathbb{A}{}_{f}^{p}\times\mathbb{Q}%
{}_{w}^{\mathrm{al}},\quad d=\dim(A).
\]
The conjecture says that it lies in $\mathbb{Q}{}\subset\mathbb{A}{}_{f}%
^{p}\times\mathbb{Q}{}_{w}^{\mathrm{al}}$. Equivalently, it says that the
$l$-component of $\gamma_{0}\cup\delta$ is a rational number independent of
$l$.

The conjecture is true for a particular $\gamma$ if $\gamma_{0}$ is algebraic.
Therefore, the conjecture is implied by the Hodge conjecture for abelian
varieties (or even by the weaker statement that the Hodge classes specialize
to algebraic classes).

\begin{example}
\label{r75}If $A$ is a CM abelian variety such that $A_{0}$ is simple and
ordinary, then the rationality conjecture holds for $A$ and its powers. To see
this, note that the hypotheses imply that $\End^{0}(A_{0})\simeq\End^{0}(A)$,
which is a CM-field of degree $2\dim A$. This isomorphism defines an
isomorphism $L(A_{0})\simeq L(A)$ of Lefschetz groups, and hence the
specialization map $H_{\mathbb{A}{}}^{2\ast}(A^{n})(\ast)\rightarrow
H_{\mathbb{A}{}}^{2\ast}(A_{0}^{n})(\ast)$ defines an isomorphism
$\mathcal{D}^{\ast}(A^{n})\simeq\mathcal{D}{}^{\ast}(A_{0}^{n})$ on the
Lefschetz classes for all $n$. In other words, every Lefschetz class $\delta$
on $A_{0}^{n}$ lifts uniquely to a Lefschetz class $\delta^{\prime}$ on
$A^{n}$, and so
\[
\gamma_{0}\cup\delta=\gamma\cup\delta^{\prime}\in\mathbb{Q}{}.
\]

\end{example}

\begin{definition}
\label{r76}Let $A$ be an abelian variety over $\mathbb{Q}{}^{\mathrm{al}}$
with good reduction to an abelian variety $A_{0}$ over $\mathbb{F}{}$. A Hodge
class $\gamma$ on $A$ is \emph{locally }$w$\emph{-Lefschetz} if its image
$\gamma_{0}$ in $H_{\mathbb{A}{}}^{2\ast}(A_{0})(\ast)$ is in the
$\mathbb{A}{}$-span of the Lefschetz classes, and it is $w$\emph{-Lefschetz}
if $\gamma_{0}$ is itself a Lefschetz class.
\end{definition}

\begin{rc2}
\label{rc2}Let $A$ be an abelian variety over $\mathbb{Q}{}^{\mathrm{al}}$
with good reduction to an abelian variety $A_{0}$ over $\mathbb{F}{}$. Every
locally $w$-Lefschetz Hodge class on $A$ is $w$-Lefschetz.
\end{rc2}

Notice that $\gamma_{0}$ is locally $w$-Lefschetz if and only if it is fixed
by $L(A_{0})$. Therefore, the conjecture asserts that a Hodge class on $A$
fixed by $L(A_{0})$ specializes to a Lefschetz class on $A_{0}$. Equivalently,
$\mathcal{B}{}^{\ast}(A)\cap\mathcal{D}{}^{\ast}(A_{0})$ is a $\mathbb{Q}{}%
$-structure on $\mathcal{B}^{\ast}(A)_{\mathbb{A}{}}\cap\mathcal{D}{}^{\ast
}(A_{0})_{\mathbb{A}{}}$ (intersections inside $H_{\mathbb{A}{}}^{2\ast}%
(A_{0})(\ast)$) (see \ref{r83} below).

\begin{theorem}
\label{r77}The following statements are equivalent:

\begin{enumerate}
\item The rationality conjecture holds for all CM abelian varieties over
$\mathbb{Q}{}^{\mathrm{al}}$.

\item The weak rationality conjecture holds for all CM abelian varieties over
$\mathbb{Q}{}^{\mathrm{al}}$.

\item There exists a good theory of rational Tate classes on $\mathcal{S}%
{}_{0}$.

\item There exists a commutative diagram of tannakian categories as in
(\ref{e18}) bound by the diagram of fundamental groups at right in (\ref{e18})
and, for every $l$ (including $l=p$) there exists a fibre functor $\omega_{l}$
on $\Mot(\mathbb{F}{})$ such that $\omega_{l}\circ R$ and $\omega_{l}\circ I$
are equal to the standard fibre functors.
\end{enumerate}
\end{theorem}

\begin{proof}
(a)$\implies$(b): Choose a $\mathbb{Q}{}$-basis $e_{1},\ldots,e_{t}$ for the
space of Lefschetz classes of codimension $r$ on $A_{0}$, and let
$f_{1},\ldots,f_{t}$ be the dual basis for the space of Lefschetz classes of
complementary dimension (here we use \cite{milne1999lc}, 5.2, 5.3). If
$\gamma$ is a locally $w$-Lefschetz class of codimension $r$, then $\gamma
_{0}=\sum c_{i}e_{i}$ for some $c_{i}\in\mathbb{A}{}$. Now%
\[
\langle\gamma_{0}\cup f_{j}\rangle=c_{j}%
\]
which the rationality conjecture implies lies in $\mathbb{Q}{}$.

(c)$\implies$(a): If there exists a good theory $\mathcal{R}{}$ of rational
Tate classes, then certainly the rationality conjecture is true, because then
$\langle\gamma_{0}\cup\delta\rangle\in\mathcal{R}{}^{\dim A}\simeq\mathbb{Q}%
{}$.

(d)$\implies$(c):We saw in Proposition \ref{r25} that a quotient functor
$q\colon\LMot(\mathbb{F})\rightarrow\mathsf{M}$ with certain properties gives
rise to a theory of rational Tate classes on $\mathcal{S}{}_{0}$. The
existence of the commutative square at the left of (\ref{e18}) implies that
the theory is good.

We shall complete the proof of the theorem in the next subsection by proving
that (b)$\iff$(d).
\end{proof}

\begin{remark}
\label{r82}Let $A$ be a CM abelian variety over $\mathbb{Q}{}^{\mathrm{al}}$.
For each $r$,
\begin{align*}
H_{\mathbb{A}{}}^{2r}(A)(r)^{L(A_{0})\cdot MT(A)}  &  \subset H_{\mathbb{A}{}%
}^{2r}(A)(r)^{MT(A)}\simeq\mathcal{B}{}^{r}(A)\otimes_{\mathbb{Q}{}}%
\mathbb{A}{}\\
H_{\mathbb{A}{}}^{2r}(A)(r)^{L(A_{0})\cdot MT(A)}  &  \subset H_{\mathbb{A}{}%
}^{2r}(A)(r)^{L(A_{0})}\simeq\mathcal{\mathcal{D}{}}{}^{r}(A_{0}%
)\otimes_{\mathbb{Q}{}}\mathbb{A}\text{.}%
\end{align*}
It follows that there are two $\mathbb{Q}{}$-structures on $H_{\mathbb{A}{}%
}^{2r}(A)(r)^{L(A_{0})\cdot MT(A)}$, namely, its intersection with
$\mathcal{B}{}^{r}(A)$ and its intersection with $\mathcal{D}{}^{r}(A_{0})$.
Conjecture \ref{rc2} is the statement that these two $\mathbb{Q}{}$-structures
are equal.
\end{remark}

\begin{remark}
\label{r83}Let $A$ be a CM abelian variety over $\mathbb{Q}{}^{\mathrm{al}}$.
For each $r$ and $\ell\neq p$,%
\[
\mathcal{B}{}^{r}(A)\otimes\mathbb{Q}{}_{\ell}\hookrightarrow H_{\ell}%
^{2r}(A)(r)\simeq H_{\ell}^{2r}(A_{0})(r)\hookleftarrow\mathcal{D}{}^{r}%
(A_{0})\otimes\mathbb{Q}{}_{\ell}.
\]
Conjecture \ref{rc2} states that $\mathcal{B}^{r}(A)\cap\mathcal{D}{}%
^{r}(A_{0})$ is a $\mathbb{Q}{}$-structure on $\left(  \mathcal{B}{}%
^{r}(A)\otimes\mathbb{Q}{}_{\ell}\right)  \cap\left(  \mathcal{D}{}^{r}%
(A_{0})\otimes\mathbb{Q}{}_{\ell}\right)  $ (for all $\ell\neq p$, and also
the analogous statement for $p$).
\end{remark}

\begin{aside}
\label{r85}It is conjectured that, in the case of good reduction, every
$\mathbb{F}{}$-point on a Shimura variety lifts to a \textit{special} point
(special lift conjecture).\footnote{This conjecture arose when the author was
extending the statement of the conjecture of Langlands and Rapoport from
Shimura varieties defined by reductive groups with simply connected derived
group to all Shimura varieties (see \cite{milne1992}). A proof of it has been
announce by \citet{vasiu2003cm}} This conjecture implies that, given an
abelian variety $A$ over $\mathbb{Q}{}^{\mathrm{al}}$ with good reduction to
an abelian variety $A_{0}$ over $\mathbb{F}{}$ and a Hodge class $\gamma$ on
$A$, there exists a CM abelian variety $A^{\prime}$ over $\mathbb{Q}%
{}^{\mathrm{al}}$ and a Hodge class $\gamma^{\prime}$ on $A^{\prime}$ for
which there exists an isogeny $A_{0}^{\prime}\rightarrow A_{0}$ sending
$\gamma_{0}^{\prime}$ to $\gamma_{0}$. From this it follows that the
rationality conjecture for CM abelian varieties implies the rationality
conjecture for all abelian varieties.
\end{aside}

\begin{aside}
Deligne (2000) notes that the following corollary of the Hodge conjecture
would be particularly interesting: let $A_{1}$ and $A_{2}$ be two liftings of
an abelian variety $A_{0}/\mathbb{F}{}$ to characteristic zero, and let
$\gamma_{1}$ and $\gamma_{2}$ be Hodge classes of complementary dimension on
$A_{1}$ and $A_{2}$; then $\left(  \gamma_{1}\right)  _{0}\cup\left(
\gamma_{2}\right)  _{0}\in\mathbb{Q}{}$. This is implied by the conjunction of
the rationality conjecture for CM abelian varieties and the special lift conjecture.
\end{aside}

\subsection{The rationality conjecture and the existence of good rational Tate
classes}

Assume, for the moment, that we have a good theory of rational Tate classes on
$\mathcal{S}{}_{0}$. Then the diagrams in (\ref{e18}) can be extended as
follows:%
\[
\begin{CD}
\mathsf{CM}^P,\omega^R @<<< \LCM^{L},\omega^{R^L}@<<< \LCM^{L\cdot S}\\
@VVV@VVV@VVV\\
\mathsf{CM} @<J<< \LCM@<<<\LCM^{S},\omega^J\\
@VV{R}V@VV{R^{L}}V@VVV\\
\Mot @<I<<\LMot@<<<\LMot^{P},\omega^I%
\end{CD}\quad\quad\minCDarrowwidth20pt \begin{CD}
S/P @>>> T/L @>>> T/L\cdot S\\
@AAA@AAA@AAA\\
S @>>> T @>>> T/S\\
@AAA@AAA@AAA \\
P @>>> L @>>> L/P.
\end{CD}
\]

\noindent Here, each of the functors $R$, $R^{L}$, $I$, and $J$ is a quotient
functor. In summary:
\begin{align*}
\Mot\left(  \mathbb{F}{}\right)    & =\CM(\mathbb{Q}^{\text{al}})/\omega
^{R}\text{ with }\omega^{R}\text{ the }\mathbb{Q}{}\text{-valued fibre functor
}X\rightsquigarrow\Hom_{\Mot}(\1,X)\text{ on }\CM(\mathbb{Q}^{\text{al}}%
)^{P};\\
\LMot(\mathbb{F}){}  & =\LCM(\mathbb{Q}{}^{\text{al}})/\omega^{R^{L}}\text{
with }\omega^{R^{L}}(X)=\Hom_{\LMot(\mathbb{F})}(\1,R^{L}X)\text{ for }X\text{
in }\LCM(\mathbb{Q}{}^{\text{al}})^{L};\\
\Mot(\mathbb{F}{})  & =\LMot(\mathbb{F})/\omega^{I}\text{ with }\omega
^{I}(X)=\Hom_{\Mot(\mathbb{F}{})}(\1,IX)\text{ for }X\text{ in }%
\LMot(\mathbb{F})^{P};\\
\CM(\mathbb{Q}^{\text{al}})  & =\LCM(\mathbb{Q}{}^{\text{al}})/\omega
^{J}\text{ with }\omega^{J}(X)=\Hom_{\mathsf{CM}(\mathbb{Q}{}^{\mathrm{al}}%
)}(\1,JX)\text{ for }X\text{ in }\LCM(\mathbb{Q}{}^{\text{al}})^{S}.
\end{align*}

For a fibre functor $\omega$ on a tannakian subcategory of $\LCM$ containing
$\LCM^{L\cdot S}$, we let $\omega|$ denote the restriction of $\omega$ to
$\LCM^{L\cdot S}$.

For $X$ in $\LCM(\mathbb{Q}{}^{\text{al}})^{L\cdot S}$,%
\[
\omega^{R^{L}}(X)\overset{\textup{{\tiny def}}}{=}\Hom_{\LMot}(\1,R^{L}%
X)\simeq\Hom_{\Mot}(\1,IR^{L}(X))\quad
\]
because $R^{L}X$ lies in $\LMot(\mathbb{F})^{L}$ and $I$ defines an
equivalence $\LMot(\mathbb{F})^{L}\rightarrow\Mot(\mathbb{F)}^{P}$ (recall
that both subcategories are canonically tensor equivalent with the category of
$\mathbb{Q}{}$-vector spaces). Similarly,%
\[
\Hom_{\Mot}(\1,IR^{L}(X))=\Hom_{\Mot}(\1,RJ(X))\simeq\Hom_{\mathsf{CM}%
}(\1,J(X))\overset{\textup{{\tiny def}}}{=}\omega^{J}(X).
\]
In fact, $\omega^{R^{L}}(X)=\omega^{J}(X)$ as subspaces of $\omega
_{\mathbb{A}{}}(X)$. Thus, $\omega^{R^{L}}|=\omega^{J}|$ as subfunctors of
$\omega_{\mathbb{A}{}}|$.

We now drop the assumption that $\Mot(\mathbb{F}{})$ exists, and we attempt to
construct it from the rest of the diagram. We want to obtain $\Mot(\mathbb{F}%
{})$ simultaneously as a quotient of $\CM$ and $\LMot$, and for this we need
$\mathbb{Q}{}$-valued fibre functors $\omega^{I}$ on $\LMot^{P}$ and
$\omega^{R}$ on $\CM^{P}$ satisfying a compatibility condition implying that
the two quotients are essentially the same.

Because the sequence%
\[
0\rightarrow S/P\rightarrow T/L\rightarrow T/\left(  S\cdot L\right)
\rightarrow0
\]
is exact (\cite{milne1999lm}, 6.1), the category $\CM^{P}$ is itself the
quotient $\LCM^{L}/\omega_{1}$ of $\LCM^{L}$ by the $\mathbb{Q}{}$-valued
fibre functor on $\LCM^{L\cdot S}$%
\[
\omega_{1}\colon X\rightsquigarrow\Hom_{\CM(\mathbb{Q}^{\text{al}}%
)}(\1,JX)=\omega^{J}(X).
\]
In other words, $\omega_{1}=\omega^{J}|$. According to (\ref{q2}), to give a
fibre functor $\omega^{R}$ on $\CM^{P}$ is the same as to give a fibre functor
$\omega$ on $\LCM(\mathbb{Q}{}^{\text{al}})^{L}$ together with an isomorphism
$\omega^{J}|\rightarrow\omega|$. In order to get a commutative diagram as in
(\ref{e18}), we must take $\omega=\omega^{R^{L}}$, and so we need an
isomorphism $\omega^{J}|\rightarrow\omega^{R^{L}}|$. In order for the standard
fibre functors to factor correctly through the quotient $\mathsf{CM}%
(\mathbb{Q}{}^{\mathrm{al}})/\omega^{R}$ we need this isomorphism to be
compatible with the canonical isomorphism of the functors $\omega
_{\mathbb{A}{}}$, or, with the identification we are making, we need the
isomorphism $\omega^{J}|\rightarrow\omega^{R^{L}}|$ to be an equality of
subfunctors of $\omega_{\mathbb{A}{}}$. In summary, we have shown:

\begin{theorem}
\label{r87}A diagram (\ref{e18}) exists, together with a functors $\omega_{l}$
on $\Mot$ such that $\omega_{l}\circ I=\omega_{l}$ and $\omega_{l}\circ
R=\omega_{l}$ for all $l$ if and only if $\omega^{J}|=\omega^{R^{L}}|$ as
subfunctors of $\omega_{\mathbb{A}{}}$ on $\LCM^{L\cdot S}$.
\end{theorem}

This completes the proof of Theorem \ref{r77}, because \textquotedblleft%
$\omega^{J}|=\omega^{R^{L}}|$ as subfunctors of $\omega_{\mathbb{A}{}}$ on
$\LCM^{L\cdot S}$\textquotedblright\ is a restatement of Conjecture \ref{rc2}
(see Remark \ref{r82}).

\begin{aside}
\label{r40}In the above, we have shown how to define $\Mot(\mathbb{F}{})$ as a
quotient of $\CM(\mathbb{Q}^{\text{al}})$. Similarly, we could have defined it
as a quotient of $\LMot(\mathbb{F})$, but, more symmetrically, we can define
it as a quotient of $\LCM(\mathbb{Q}{}^{\text{al}})$ or of $\CM(\mathbb{Q}%
^{\mathrm{al}})\otimes\LMot(\mathbb{F}{})$.
\end{aside}

\subsection{Ordinary abelian varieties}

Let $\CM^{\prime}(\mathbb{Q}^{\text{al}})$ and $\LCM^{\prime}(\mathbb{Q}%
{}^{\text{al}})$ be the tannakian subcategories generated by CM abelian
varieties over $\mathbb{Q}{}^{\mathrm{al}}$ specializing to simple ordinary
abelian varieties over $\mathbb{F}{}$. Because the rationality conjecture
holds for such abelian varieties (see \ref{r75}), we obtain unconditionally a
good theory of rational Tate classes on ordinary abelian varieties over
$\mathbb{F}{}$. Moreover, we obtain a canonical commutative diagram%
\[
\begin{CD}
\mathsf{CM}^{\prime}(\mathbb{Q}^{\mathrm{al}})@<J<<\LCM^{\prime}(\mathbb{Q}^{\text{al}})\\
@VVRV@VV{R^{\prime}}V\\
\Mot^{\mathrm{ord}}(\mathbb{F})@<I<<\LMot^{\mathrm{ord}}(\mathbb{F})
\end{CD}
\]
in which $\Mot^{\mathrm{ord}}(\mathbb{F}{})$ and $\LMot^{\mathrm{ord}%
}(\mathbb{F}{})$ are generated by the ordinary abelian varieties over
$\mathbb{F}{}$. For each prime $l$, there exists a fibre functor $\omega_{l}$
on $\Mot^{\mathrm{ord}}(\mathbb{F}{})$ such that $\omega_{l}\circ R=\omega
_{l}$ and $\omega_{l}\circ I=\omega_{l}$. In this case, the functors $R$ and
$R^{\prime}$ are tensor equivalences, and so there is a canonical
$\mathbb{Q}{}$-valued fibre functor on $\Mot^{\mathrm{ord}}(\mathbb{F}{}%
)$.\footnote{The mere existence of a $\mathbb{Q}{}$-valued fibre functor on
the category of motives generated by ordinary abelian varieties is not hard to
prove assuming the Tate conjecture. It amounts to showing that the class of
the category in $H^{2}$ is zero, but this follows from the analogue of
statement (*) following 3.15 in \cite{milne1994} for the ordinary version of
$P$. (The proof of (*) is clarified in my preprint Motives over $\mathbb{F}%
{}_{p}$ --- see Proposition 1.1). Moreover, simply by looking at the
cohomology of the groups involved, it is possible to show that there exists a
$\mathbb{Q}{}$-valued fibre functor that becomes isomorphic to the standard
$\ell$-adic fibre functor when tensored with $\mathbb{Q}{}_{\ell}$.} In other
words, as expected, ordinary abelian varieties and their motives in
characteristic $p$ behave very much as their counterparts in characteristic zero.

\bibliographystyle{cbe}
\bibliography{../../refs}

\bsmall\bigskip\noindent Mathematics Department, University of Michigan, Ann
Arbor, MI 48109, USA

\noindent Email: \url{jmilne@umich.edu} 

\noindent Webpage: \url{www.jmilne.org/math/} \esmall

\end{document}

%% file: defna.tex
\let\cite=\citealt

\setdefaultitem{$\diamond$\hspace{0.07in}}{}{}{}

\newcommand\bquote{\begin{quote}}
\newcommand\equote{\end{quote}}
\newcommand\bcomment{}
\newcommand\bsmall{\begin{small}}
\newcommand\esmall{\end{small}}
\newcommand\bfootnotesize{\begin{footnotesize}}
\newcommand\efootnotesize{\end{footnotesize}}

\newcommand{\textnf}{\textnormal}

\renewcommand{\bar}{\overline}

\renewcommand{\Gamma}{\varGamma}
\renewcommand{\Pi}{\varPi}
\renewcommand{\Sigma}{\varSigma}

\def\1{{1\mkern-7mu1}}

\DeclareMathOperator{\Aut}{Aut}

\DeclareMathOperator{\End}{End}

\DeclareMathOperator{\Gal}{Gal}
\DeclareMathOperator{\GL}{GL}

\DeclareMathOperator{\Hom}{Hom}
\DeclareMathOperator{\id}{id}

\DeclareMathOperator{\Ker}{Ker}

\DeclareMathOperator{\MT}{MT}

\DeclareMathOperator{\SL}{SL}

\DeclareMathOperator{\Tr}{Tr}

\newcommand{\CM}{\mathsf{CM}}

\newcommand{\LCM}{\mathsf{LCM}}
\newcommand{\LMot}{\mathsf{LMot}}

\newcommand{\Mot}{\mathsf{Mot}}

\newcommand{\Vc}{\mathsf{Vec}}



%% file: f1.tex
\[\bfig
\node a(0,400)[T]
\node b(500,400)[Q]
\node c(0,200)[\rotatebox{90}{$\subset$}]
\node d(500,200)[\rotatebox{90}{$\subset$}]
\node e(1000,0)[\mathsf{Vec}_k]
\node f(0,0)[T^q]
\node g(500,0)[Q^\pi]
\arrow[a`b;{q}]
\arrow/{}/[c`d;{}]
\arrow[f`g;{q|T^q}]
\arrow[g`e;{\gamma^Q}]
\arrow|m|/{@{>}@/_15pt/}/[f`e;{\omega^q}]
\efig
\qquad \Hom(qX,qY)\simeq\omega^{q}(\underline{\Hom}(X,Y)^{H}).\]

%% file: f2.tex
\[\bfig
\node a(0,400)[T]
\node b(500,400)[Q]
\node c(0,200)[\bigcup]
\node e(1000,400)[\mathsf{Vec}_k]
\node f(0,0)[T^q]
\arrow[a`b;{q}]
\arrow[b`e;{\omega}]
\arrow|m|/{@{>}@/^15pt/}/[a`e;{\omega^\prime}]
\arrow/{}/[c`e;{}]
\arrow|m|[f`e;{\omega^q}]
\efig
\qquad
\bfig
\morphism(0,200)|a|/{@{>}@/^1em/}/<500,0>[T^q`\mathsf{Vec}_k;\omega^\prime|T^q]
\morphism(0,200)|b|/{@{>}@/_1em/}/<500,0>[T^q`\mathsf{Vec}_k;\omega^q]
\morphism(250,150)|r|/=>/<0,100>[{}`{};a]
\efig
\]

%% file: rtc.bbl
\begin{thebibliography}{}

\bibitem[\protect\astroncite{Andr{\'e}}{1996}]{andre1996}
{\sc Andr{\'e}, Y.} 1996.
\newblock Pour une th{\'e}orie inconditionnelle des motifs.
\newblock {\em Inst. Hautes {\'E}tudes Sci. Publ. Math.} pp. 5--49.

\bibitem[\protect\astroncite{Andr{\'e}}{2006}]{andre2006b}
{\sc Andr{\'e}, Y.} 2006.
\newblock Cycles de {T}ate et cycles motiv\'es sur les vari\'et\'es
  ab\'eliennes en caract\'eristique {$p>0$}.
\newblock {\em J. Inst. Math. Jussieu} 5:605--627.

\bibitem[\protect\astroncite{Clozel}{1999}]{clozel1999}
{\sc Clozel, L.} 1999.
\newblock Equivalence num\'erique et \'equivalence cohomologique pour les
  vari\'et\'es ab\'eliennes sur les corps finis.
\newblock {\em Ann. of Math. (2)} 150:151--163.

\bibitem[\protect\astroncite{Clozel}{2008}]{clozel2008}
{\sc Clozel, L.} 2008.
\newblock Equivalence num\'erique, \'equivalence cohomologique, et th\'eorie de
  {L}efschetz des vari\'et\'es ab\'eliennes sur les corps finis.
\newblock Preprint; not available on the web.

\bibitem[\protect\astroncite{Deligne}{1982}]{deligne1982}
{\sc Deligne, P.} 1982.
\newblock Hodge cycles on abelian varieties (notes by {J}.{S}. {M}ilne), pp.
  9--100.
\newblock {\em In} Hodge cycles, motives, and {S}himura varieties, Lecture
  Notes in Mathematics. Springer-Verlag, Berlin.
\newblock Version with endnotes available at www.jmilne.org/math/Documents.

\bibitem[\protect\astroncite{Deligne}{1990}]{deligne1990}
{\sc Deligne, P.} 1990.
\newblock Cat{\'e}gories tannakiennes, pp. 111--195.
\newblock {\em In} The Grothendieck Festschrift, Vol.\ {II}, Progr. Math.
  Birkh{\"a}user Boston, Boston, MA.

\bibitem[\protect\astroncite{Deligne}{2000}]{deligne2000}
{\sc Deligne, P.} 2000.
\newblock The {H}odge conjecture.
\newblock \url{www.claymath.org/millennium/Hodge_Conjecture/hodge.pdf}; printed
  in The millennium prize problems, 45--53, Clay Math. Inst., Cambridge, MA,
  2006.

\bibitem[\protect\astroncite{Deligne and Milne}{1982}]{deligneM1982}
{\sc Deligne, P. and Milne, J.~S.} 1982.
\newblock Tannakian categories, pp. 101--228.
\newblock {\em In} Hodge cycles, motives, and {S}himura varieties, Lecture
  Notes in Mathematics 900. Springer-Verlag, Berlin.

\bibitem[\protect\astroncite{Grothendieck}{1968}]{grothendieck1968}
{\sc Grothendieck, A.} 1968.
\newblock Le groupe de {B}rauer. {III}. {E}xemples et compl{\'e}ments, pp.
  88--188.
\newblock {\em In} Dix Expos{\'e}s sur la Cohomologie des Sch{\'e}mas.
  North-Holland, Amsterdam.
\newblock Available at \url{www.grothendieck-circle.org}.

\bibitem[\protect\astroncite{Hazama}{2002}]{hazama2002}
{\sc Hazama, F.} 2002.
\newblock General {H}odge conjecture for abelian varieties of {CM}-type.
\newblock {\em Proc. Japan Acad. Ser. A Math. Sci.} 78:72--75.

\bibitem[\protect\astroncite{Hazama}{2003}]{hazama2003}
{\sc Hazama, F.} 2003.
\newblock On the general {H}odge conjecture for abelian varieties of {CM}-type.
\newblock {\em Publ. Res. Inst. Math. Sci.} 39:625--655.

\bibitem[\protect\astroncite{Jacobson}{1943}]{jacobson1943}
{\sc Jacobson, N.} 1943.
\newblock The {T}heory of {R}ings.
\newblock American Mathematical Society Mathematical Surveys, vol. I. American
  Mathematical Society, New York.

\bibitem[\protect\astroncite{Jannsen}{1992}]{jannsen1992}
{\sc Jannsen, U.} 1992.
\newblock Motives, numerical equivalence, and semi-simplicity.
\newblock {\em Invent. Math.} 107:447--452.

\bibitem[\protect\astroncite{Katz and Messing}{1974}]{katzM1974}
{\sc Katz, N.~M. and Messing, W.} 1974.
\newblock Some consequences of the {R}iemann hypothesis for varieties over
  finite fields.
\newblock {\em Invent. Math.} 23:73--77.

\bibitem[\protect\astroncite{Kleiman}{1968}]{kleiman1968}
{\sc Kleiman, S.~L.} 1968.
\newblock Algebraic cycles and the {W}eil conjectures, pp. 359--386.
\newblock {\em In} Dix espos\'es sur la cohomologie des sch\'emas.
  North-Holland, Amsterdam.

\bibitem[\protect\astroncite{Kleiman}{1994}]{kleiman1994}
{\sc Kleiman, S.~L.} 1994.
\newblock The standard conjectures, pp. 3--20.
\newblock {\em In} Motives (Seattle, WA, 1991), volume~55 of {\em Proc. Sympos.
  Pure Math.} Amer. Math. Soc., Providence, RI.

\bibitem[\protect\astroncite{Langlands and Rapoport}{1987}]{langlandsR1987}
{\sc Langlands, R.~P. and Rapoport, M.} 1987.
\newblock Shimuravariet{\"a}ten und {G}erben.
\newblock {\em J. Reine Angew. Math.} 378:113--220.
\newblock Available online at the {\sc Langlands Archive}.

\bibitem[\protect\astroncite{Milne}{1992}]{milne1992}
{\sc Milne, J.~S.} 1992.
\newblock The points on a {S}himura variety modulo a prime of good reduction,
  pp. 151--253.
\newblock {\em In} The zeta functions of Picard modular surfaces. Univ.
  Montr{\'e}al, Montreal, QC.

\bibitem[\protect\astroncite{Milne}{1994}]{milne1994}
{\sc Milne, J.~S.} 1994.
\newblock Motives over finite fields, pp. 401--459.
\newblock {\em In} Motives (Seattle, WA, 1991), Proc. Sympos. Pure Math. Amer.
  Math. Soc., Providence, RI.

\bibitem[\protect\astroncite{Milne}{1999a}]{milne1999lc}
{\sc Milne, J.~S.} 1999a.
\newblock Lefschetz classes on abelian varieties.
\newblock {\em Duke Math. J.} 96:639--675.

\bibitem[\protect\astroncite{Milne}{1999b}]{milne1999lm}
{\sc Milne, J.~S.} 1999b.
\newblock Lefschetz motives and the {T}ate conjecture.
\newblock {\em Compositio Math.} 117:45--76.

\bibitem[\protect\astroncite{Milne}{2002a}]{milne2002}
{\sc Milne, J.~S.} 2002a.
\newblock Polarizations and {G}rothendieck's standard conjectures.
\newblock Original version of \cite{milne2002p}; available at the author's
  webpage.

\bibitem[\protect\astroncite{Milne}{2002b}]{milne2002p}
{\sc Milne, J.~S.} 2002b.
\newblock Polarizations and {G}rothendieck's standard conjectures.
\newblock {\em Ann. of Math. (2)} 155:599--610.

\bibitem[\protect\astroncite{Milne}{2006}]{milneCM}
{\sc Milne, J.~S.} 2006.
\newblock Complex Multiplication.
\newblock Available at the author's webpage.

\bibitem[\protect\astroncite{Milne}{2007a}]{milne2007qtc}
{\sc Milne, J.~S.} 2007a.
\newblock Quotients of {T}annakian categories.
\newblock {\em Theory Appl. Categ.} 18:No. 21, 654--664.

\bibitem[\protect\astroncite{Milne}{2007b}]{milne2007aim}
{\sc Milne, J.~S.} 2007b.
\newblock The {T}ate conjecture over finite fields ({AIM} talk).
\newblock Available at the author's webpage and at arXive:0709.3040.

\bibitem[\protect\astroncite{Milne and Ramachandran}{2004}]{milneR2004}
{\sc Milne, J.~S. and Ramachandran, N.} 2004.
\newblock Integral motives and special values of zeta functions.
\newblock {\em J. Amer. Math. Soc.} 17:499--555.

\bibitem[\protect\astroncite{Saavedra~Rivano}{1972}]{saavedra1972}
{\sc Saavedra~Rivano, N.} 1972.
\newblock Cat{\'e}gories {T}annakiennes.
\newblock Springer-Verlag, Berlin.

\bibitem[\protect\astroncite{Serre and Tate}{1968}]{serreT1968}
{\sc Serre, J.-P. and Tate, J.} 1968.
\newblock Good reduction of abelian varieties.
\newblock {\em Ann. of Math. (2)} 88:492--517.

\bibitem[\protect\astroncite{Tate}{1994}]{tate1994}
{\sc Tate, J.~T.} 1994.
\newblock Conjectures on algebraic cycles in {$l$}-adic cohomology, pp. 71--83.
\newblock {\em In} Motives (Seattle, WA, 1991), volume~55 of {\em Proc. Sympos.
  Pure Math.} Amer. Math. Soc., Providence, RI.

\bibitem[\protect\astroncite{Vasiu}{2003}]{vasiu2003cm}
{\sc Vasiu, A.} 2003.
\newblock {CM}-lifts of isogeny classes of {S}himura {F}-crystals over finite
  fields.
\newblock arXiv:math/0304128.

\bibitem[\protect\astroncite{Weil}{1948}]{weil1948}
{\sc Weil, A.} 1948.
\newblock Vari{\'e}t{\'e}s ab{\'e}liennes et courbes alg{\'e}briques.
\newblock Actualit{\'e}s Sci. Ind., no. 1064 = Publ. Inst. Math. Univ.
  Strasbourg 8 (1946). Hermann {\&} Cie., Paris.

\bibitem[\protect\astroncite{Zink}{1983}]{zink1983}
{\sc Zink, T.} 1983.
\newblock Isogenieklassen von {P}unkten von {S}himuramannigfaltigkeiten mit
  {W}erten in einem endlichen {K}\"orper.
\newblock {\em Math. Nachr.} 112:103--124.

\end{thebibliography}
